\documentclass[11pt]{article}

\usepackage[utf8]{inputenc}
\usepackage[english]{babel}
\usepackage[titletoc]{appendix}
\usepackage[centertags]{amsmath}
\usepackage{amsthm, amssymb}
\usepackage{dsfont}
\usepackage{bbm}
\usepackage{graphicx}
\usepackage{pgf,tikz}
\usepackage{caption}

\usepackage[pdftex,
            pdfauthor={Max Grieshammer},
            pdftitle={Family size decomposition of genealogical trees},
            pdfproducer={pdfLaTex},
            pdfcreator={}]{hyperref}

\usepackage{fancyhdr}
\newtheorem{theorem}{Theorem}[section]
\newtheorem{proposition}[theorem]{Proposition}
\newtheorem{corollary}[theorem]{Corollary}
\newtheorem{lemma}[theorem]{Lemma}
\newtheorem{assumption}[theorem]{Assumption}

\makeatletter                                           
\def\th@newremark{\th@remark\thm@headfont{\bfseries}}   
\makeatletter                                           

\theoremstyle{definition}
\newtheorem{definitionx}[theorem]{Definition}

\theoremstyle{newremark}
\newtheorem{remarkx}[theorem]{Remark}

\newtheorem{examplex}[theorem]{Example}

\newenvironment{example}
  {\pushQED{\qed}\examplex}
  {\vspace{0.5cm}\popQED\endexamplex}

\newenvironment{remark}
  {\pushQED{\qed}\remarkx}
  {\vspace{0.5cm}\popQED\endremarkx}

\newenvironment{definition}
  {\pushQED{\qed}\definitionx}
  {\vspace{0.5cm}\popQED\enddefinitionx}

 \numberwithin{equation}{section}

\newcommand{\eps}{\varepsilon}
\newcommand{\e}{\varepsilon}
\def\h{h}

\def\N{\mathbb N}

\def\B{\mathcal B}
\def\M{\mathcal M}

\def\R{\mathbb R}

\def\UM{\mathbb{U}}							 	 							
\def\UMatom{\UM^a}
\def\UMcont{\UM^c}
\def\UMatomc{\UM^a_c}

\def\UMc{\UM_c}
\def\UMI{\mathbb I}
\def\UMB{\mathbb B}
\def\domF{\mathbb F}
\def\MM{\mathbb{M}}

\def\mfu{\mathfrak{u}}						 							
	 	
\def\mfm{\mathfrak{m}} 				
\def\Sk{\mathfrak{v}}

\def\rep{\mathfrak{r}}													
\def\repp{\mathfrak{q}}

\def\L{\mathcal L}
			 							
\def\U{\mathcal{U}}

\def\Cut{\Phi}							  	   					
\def\CutTwo{\hat \Phi}						
\def\Num{\mathfrak{N}}												 
\def\num{n}												  
\def\F{\mathfrak{F}}												 	
\def\NU{\mathcal{V}}
\def\f{\mathfrak{f}}												  
\def\Top{\Psi}

\def\dSK{d^{\text{SK}}} %
\def\dSKSeq{d^{\text{SK},1}}
\def\dSKMax{d^{\text{SK},\infty}}

\def\dSeq{d^1} 
\def\dMax{d^\infty} 
\def\dPr{d_{\text{Pr}}}										
\def\dGP{d_{\text{GPr}}}										
\def\dGPA{d_{\text{GPa}}}

\def\supp{\textrm{\normalfont supp}}

\newcommand{\norm}[1]{\big|\hspace{-0.05cm}\big|#1\big|\hspace{-0.05cm}\big|}

\allowdisplaybreaks[4]

\begin{document}

 \title{Family size decomposition of genealogical trees}
 \author{Max Grieshammer\footnote{Institute for Mathematics, Friedrich-Alexander Universität Erlangen-Nürnberg, Germany;  	max.grieshammer@math.uni-erlangen.de, 
 MG was supported by DFG-Grant GR-876-17.1 of A. Greven}
 }
 
 \maketitle
 
\thispagestyle{empty}

\begin{abstract}
We study the path of family size decompositions of varying depth of genealogical trees. We prove that this decomposition 
as a function on (equivalence classes of) ultra-metric measure spaces to the Skorohod space describing the family sizes
at different depths is perfect onto its image, i.e. there is a suitable topology such that this map is continuous closed 
surjective and pre-images of compact sets are compact. We also specify a (dense) subset so 
that the restriction of the function to this subspace is a homeomorphism. This property allows us to argue that the 
whole genealogy of a Fleming-Viot process with mutation and selection as well as the genealogy in a Feller branching population can be 
reconstructed by the genealogical distance of two randomly chosen individuals.   
\end{abstract}

\medskip 



\newpage

\tableofcontents
\thispagestyle{empty}

\section{Introduction}

There are several approaches to study genealogical properties of a Wright-Fisher population. 
For example, one can use the Kingman coalescent (see \cite{kingman1982coalescent}) which naturally generates the 
genealogical tree of a neutral Moran model or a neutral Wright-Fisher population at a fixed time. When selection is present, things get harder. 
This is because in contrast to the neutral model, 
the tree of the current population now depends on the whole type evolution until the present time (see Theorem 2 in \cite{DGP12}). 
Nevertheless, one can use, for example, the ancestral selection graph introduced by Krone and Neuhauser (see \cite{Krone} and \cite{neuhauser1997genealogy})
or the lookdown construction introduced by Donelly and Kurtz (see \cite{donnelly1996countable} and \cite{donnelly1999genealogical}), to construct or read off 
genealogical properties. But still, it is quite hard to get explicit results on the genealogy. \par 
Depperschmidt, Greven, Pfaffelhuber, Winter \cite{GPW09}, (compare also \cite{GPW13}; and \cite{DGP12} for a survey) followed a different approach. 
They constructed the genealogy dynamically as a tree-valued process (we explain the notion ``tree'' in more details below). This
has the advantage, that one can use the generator 
of this Markov process to get, for example, recurrence relations of the genealogical distance of two randomly chosen 
individuals in equilibrium. \\

Here we try to connect the ``classical'' approach of coalescent models and the setting in Depperschmidt, Greven, Pfaffelhuber and
Winter. Namely we ask for a quantity that both, (exchangeable) coalescents and (genealogical) trees, have in common.  \par
Recall, that one can use the Kingman coalescent to construct the genealogy in the neutral case and that the law of the Kingman coalescent 
is determined by the law of its block frequencies via Kingman's paint box construction (see for example \cite{B}).  The process associated with the evolution 
of the block frequencies is usually called mass coalescent. We will show that a tree naturally contains the concept of ``mass coalescent'', 
which we call {\it path of family size decompositions} (or {\it family size decomposition} for simplicity) in this context. Our goal in this paper is to study properties of this decomposition 
and how to apply these results to (large) Wright-Fisher populations and Feller branching populations to gain information about genealogy. \\

To get a bit more precise (see Section \ref{sec.mms} and Section \ref{sec.decomp} for all details), we consider the space $\UM$ of (equivalence classes of) ultra-metric
measure spaces $(X,r,\mu)$, where we interpret $X$ as a set of individuals, $r$ as a genealogical distance and $\mu$ as a sampling measure. We assume that $(X,r)$ is complete and 
separable and note that one can map $(X,r)$ isometrically to the leaves of a rooted $\mathbb R$-tree (see Remark 2.7 in \cite{DGP12} and Remark 2.2 in \cite{DGP11}), which 
justifies the name {\it tree} for an ultra-metric measure space. \par 
Now, we decompose $X$ into balls $\bar B_h(x) := \{y \in X:\ r(x,y) \le h\}$ for some  $h > 0$ and note that $\bar B_h(x) \cap \bar B_h(y) = \emptyset$ for $r(x,y) > h$, since $r$ is an ultra-metric,
and the number of balls needed to cover $X$ is countable, since $X$ separable. 
We can interpret $\bar B_h(x)$ as a family descending from an ancestor who lived at the time $h$ (measured backwards) (see Figure \ref{fig.F}). 
When we now calculate the sizes, $\mu(B_h(x))$, of the different families and denote the size ordered vector by 
$\underline{a}(h):=(a_1(h),a_2(h),\ldots)$, i.e. $a_k(h) \ge a_{k+1}(h)$ for all $k$, then we finally get the notion of 
family size decomposition of trees:

\begin{itemize}
 \item[]  We call the function $\F:\UM \to D((0,\infty),\mathcal S^\downarrow)$ that maps an ultra-metric measure space to the (cadlag) 
 function $h \mapsto \underline{a}(h)$ {\it family size decomposition}.
\end{itemize}
Here, 
\begin{equation}
 \mathcal S^\downarrow :=\left\{(x_1,x_2,\ldots) \in [0,\infty)^\N:\ 
\sum_{i \in \N} x_i < \infty,\ x_1 \ge x_2 \ge \ldots \right\}
\end{equation}

\begin{figure}[ht]
\begin{tikzpicture}[scale = 0.44]
\draw (4,3)-- (6,1);
\draw (6,1)-- (6,3);
\draw (6,1)-- (8,3);
\draw (10,3)-- (11,2);
\draw (11,2)-- (12.,3);
\draw (14,3)-- (12.5,-1);
\draw (12.5,-1)-- (11,2);
\draw (16,3)-- (13.2,-2);
\draw (13.2,-2)-- (12.5,-1);
\draw (13.2,-2)-- (11,-4);
\draw (11,-4)-- (6.,1);
\draw (18,3)-- (20,1);
\draw (20,1)-- (20,3);
\draw (20,1)-- (22,3);
\draw (24,3)-- (25,2);
\draw (25,2)-- (26,3);
\draw (25,2)-- (25,0);
\draw (28,3)-- (28,0);
\draw (30,3)-- (30,0);
\draw (20,1)-- (20,0);
\draw [dash pattern=on 8pt off 8pt] (25,0)-- (26.5,-1);
\draw [dash pattern=on 8pt off 8pt] (26.5,-1)-- (28,0);
\draw [dash pattern=on 8pt off 8pt] (30,0)-- (28,-2);
\draw [dash pattern=on 8pt off 8pt] (28,-2)-- (26.5,-1);
\draw [dash pattern=on 8pt off 8pt] (20,0)-- (25,-4);
\draw [dash pattern=on 8pt off 8pt] (25,-4)-- (28,-2);
\draw [dash pattern=on 8pt off 8pt] (3,0)-- (31,0);
\draw (31.5,0) node {$\mathbf{h}$};
\begin{scriptsize}
\draw [fill=black] (4,3) circle (2.0pt);
\draw [fill=black] (6,3) circle (2.0pt);
\draw [fill=black] (8,3) circle (2.0pt);
\draw [fill=black] (10,3) circle (2.0pt);
\draw [fill=black] (12,3) circle (2.0pt);
\draw [fill=black] (14,3) circle (2.0pt);
\draw [fill=black] (16,3) circle (2.0pt);
\draw [fill=black] (18,3) circle (2.0pt);
\draw [fill=black] (20,3) circle (2.0pt);
\draw [fill=black] (22,3) circle (2.0pt);
\draw [fill=black] (24,3) circle (2.0pt);
\draw [fill=black] (26,3) circle (2.0pt);
\draw [fill=black] (28,3) circle (2.0pt);
\draw [fill=black] (30,3) circle (2.0pt);
\end{scriptsize}
\end{tikzpicture}
\caption{\label{fig.F}\footnotesize On the left side we draw an ultra-metric measure space $(X,r,\mu)$, where $|X| = 7$ and 
$\mu(\{x\}) = 1$ for all $x \in X$. We can decompose this tree into four disjoint (closed) balls of radius $h > 0$ (drawn on the right side).} 
\end{figure}
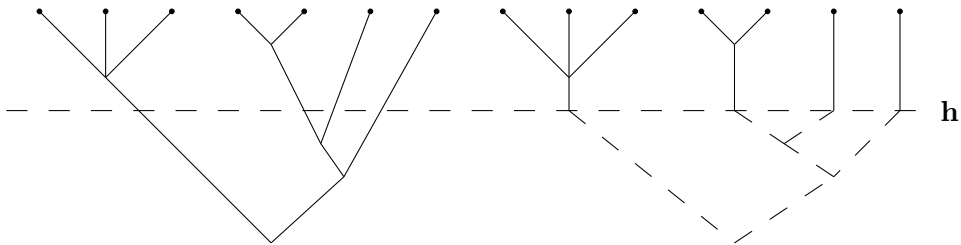

Our first goal in this paper is to study properties of this map and we can summarize our results as follows: 

\begin{itemize}
 \item[] {\it The function $\F$ is continuous, closed and preimages of compact sets are compact, i.e. $\F$ is perfect. }
\end{itemize}

Of course, we have to define suitable topologies on the respective spaces $\UM$ and $D((0,\infty),\mathcal S^\downarrow)$, so that the result is valid.
While we equip the space of cadlag functions $h \mapsto \underline{a}(h)$ with the Skorohod topology, we need to specify the topology on $\UM$. 
Typically, one would equip $\UM$ with the so called Gromov-weak topology. Convergence in this topology is equivalent to convergence of the 
corresponding distance matrix distributions $\nu^{m,(X,r,\mu)}$, where
\begin{equation}
 \nu^{m,(X,r,\mu)}(\cdot) := \mu^{\otimes m}(\{x_1,\ldots,x_{m}: (r(x_i,x_j))_{1 \le i < j \le m} \in \cdot\})
\end{equation}
(see Section \ref{sec.mms} for details) but we point out that the function $\F$ is {\it not} continuous in this topology. 
This is the reason why we need to introduce a finer topology which we call {\it Gromov-weak atomic topology}. \par 
Convergence of a sequence $\mfu_n \in \UM$, $n \in \mathbb N$ to a limit object $\mfu \in \UM$ in this new topology
is equivalent to convergence of the distance matrix distribution, i.e. convergence in the Gromov-weak topology, and convergence of 
the following quantities
\begin{itemize}
 \item[(a)] $(\nu^{2,\mfu_n})^\ast := \sum_{h \ge 0} \nu^{2,\mfu_n}(\{h\})^2 \delta_h \Rightarrow (\nu^{2,\mfu})^\ast$ as $n \rightarrow \infty$, where ``$\Rightarrow $''
 denotes the convergence in the weak topology on finite measures. 
 \item[(b)] $\nu^{2,\mfu_n}(\{0\})\rightarrow \nu^{2,\mfu}(\{0\})$ as $n \rightarrow \infty$.
\end{itemize}

The following example shows the differences between the convergence in the Gromov-weak and Gromov-weak atomic topology (see Section \ref{sec.mms} for all details).

\begin{example}(Convergence in the Gromov-weak atomic topology)\label{ex1}
	We consider the sequence  $(\{x_1,x_2,x_3\},r_n,\delta_{x_1} + \delta_{x_2} + \delta_{x_3})$,
	with 
	\begin{equation}
	\begin{split}
		&r_n(x_1,x_2) = 1, \\
		&r_n(x_1,x_3) = r_n(x_2,x_3) = 1+\frac{1}{n},\\
		&r_n(x_i,x_i) = 0 , \ i = 1,2,3
	\end{split}
	\end{equation} 
	for $n \ge 1$ and the ultra-metric measure space $(\{x_1,x_2,x_3\},r,\delta_{x_1} + \delta_{x_2} + \delta_{x_3})$, where 
	\begin{equation}
		r(x_i,x_j) = 1,\ i\not=j,\quad r(x_i,x_j) = 0,\ i = j.
	\end{equation}
		then $\mfu_n \rightarrow \mfu$ in the Gromov-weak topology (see Figure \ref{f.conv}) but $\mfu_n \not\rightarrow \mfu$ not in the 
		Gromov-weak atomic topology.\\

	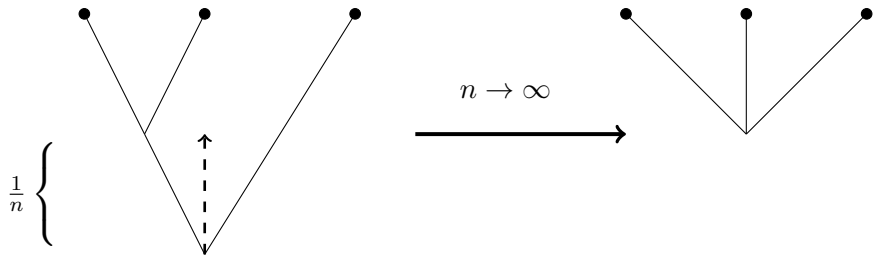
\begin{figure}[ht]
		\centering
		\begin{tikzpicture}[scale=0.8]
			\draw (5,3)-- (6,1);
			\draw (7,3)-- (6,1);
			\draw (9.5,3)-- (7,-1);
			\draw (7,-1)-- (6,1);
			\draw (14,3)-- (16,1);
			\draw (18,3)-- (16,1);
			\draw (16,3)-- (16,1);
			\draw (4.5,0)  node {$\frac{1}{n} \left\{\begin{array}{r} 
			\phantom{a} \\[0.5cm]
			\phantom{a}  \end{array}\right.$};
			\draw [->,line width=1.5pt] (10.5,1) -- (14,1);
			\draw [->,line width=1pt,dash pattern=on 4pt off 4pt] (7,-1) -- (7,1);
			\draw (12,1.7) node {$n \rightarrow \infty$};
			\begin{scriptsize}
				\draw [fill=black] (5,3) circle (2.5pt);
				\draw [fill=black] (7,3) circle (2.5pt);
				\draw [fill=black] (9.5,3) circle (2.5pt);
				\draw [fill=black] (14,3) circle (2.5pt);
				\draw [fill=black] (16,3) circle (2.5pt);
				\draw [fill=black] (18,3) circle (2.5pt);
			\end{scriptsize}	
		\end{tikzpicture}
		\caption{\label{f.conv} \footnotesize Convergence in the Gromov-weak but not in the Gromov-weak atomic topology.}
	\end{figure}
\end{example}

In view of the above example we can interpret convergence in the  Gromov-weak atomic topology  as convergence in the Gromov-weak topology plus 
some additional conditions on the convergence of the ``branching points of the tree''.  \\

Even though an ultra-metric measure space $(X,r,\mu)$ can not be reconstructed by the value $\F((X,r,\mu))$ in general (since $\F$ is not injective), we may hope that, as in the case of the Kingman coalescent,
we can find a ``nice'' subspace on which a reconstruction is possible. Indeed we prove that 
\begin{itemize}
\item[] {\it there is a dense $G_\delta $ subset of $\UM$ such that $\F$ restricted to this subspace is a homeomorphism (onto its image)} 
\end{itemize}
and call elements $\mfu$ of this subspace {\it identifiable by family sizes}.  \\

Next we want to apply our results on the tree-valued Feller diffusion and the tree-valued Fleming-Viot process and note that 
even though both processes are different, from a genealogical perspective they look quite similar: One can show that conditioned on the total mass, the genealogy in a Feller diffusion is a time inhomogeneous tree-valued Fleming-Viot process. We do not want to go into detail but refer to \cite{infdiv}, \cite{Gloede2013}, Chapter 5 or\cite{DGFeller}. The important observation is, that it is enough to consider the tree-valued Fleming-Viot process to gain genealogical information for the tree-valued Feller diffusion and vice versa. Having that in mind, we get the following result: 
\begin{itemize}
 \item[] {\it The genealogy, denoted by $\mathcal U_t$ (as an $\UM$-valued random variable), in a (large) Wright-Fisher population (with or without selection) 
 at time $t> 0$ is completely determined by the distribution of the  distance of two randomly  chosen individuals, i.e. $\mathcal L(\nu^{2,\mathcal U_t})$.}
\end{itemize}

Note that due to the above observation, this result stays valid in the situation of a tree-valued Feller diffusion. \par
In fact, we think that in the "most" infinite population models  the genealogy at a given time $t$ is identifiable by family sizes,
where we have processes in mind that arise as large population limits of graphically constructed finite population models equipped with the uniform distribution on the set of individuals.  The reason  is Proposition \ref{p.UMI.as} that mainly says that, 
whenever the vector of family sizes at some depth is absolutely continuous to the (product-)Lebesgue measure, we can reconstruct the 
whole genealogy by the path of family size decompositions of varying depth. \par 
 We also point out that the above result is not a probabilistic result in the sense that we use probabilistic arguments to reconstruct the law of the genealogy, but rather a result about 
the ``states'' of the genealogy, i.e. the genealogy realizes its values (with probability one) in a subspace of $\UM$ on which the function $\nu^{2,\cdot}$ or $\F$ is a homeomorphism (onto its image).  \par
Although we think that the function $\F$ (and the concept of family size decompositions) is an interesting object itself it can also be used 
to construct compact subsets of $\UM$ and therefore gives a tool to prove compact containment conditions for evolving genealogies (see Corollary \ref{c.perfect}). 
To be a bit more precise we have the following application in mind: (1) Construct a tree-valued process $\U^N$ via a graphical construction.
(2) Show that it has a large population limit, i.e. $\U^N \Rightarrow \U$, when $\UM$ is equipped with the Gromov-weak atomic topology. (3) Use the continuous mapping theorem to deduce that $\F(\U^N) \Rightarrow \F(\U)$ (see also Corollary \ref{c.perfect}). (4) Assume that the process $\U$ is again indexed, i.e. $\U = \U^M$, for some $M = 1,2,\ldots$ and we want to study the behavior of $\U^M$ for $M \to \infty$. 
The key idea in doing that is to observe that for an evolving genealogy $\U = (\U_t)_{t \ge 0}$ we can not only consider $\F(\U_t) = (\f(\U_t,h))_{h \ge 0}$ for fixed $t$ (which corresponds to a backward in time picture), but also $(\f(\U_t,t+h))_{t \ge 0}$ for fixed $h$ (which corresponds to a forward in time picture). It will turn out that a combination of both a backward as well as a forward in time picture can be used to get tightness of the evolving genealogies in terms of simpler processes. Roughly speaking, it is enough to prove convergence of $((\f(\U_t^M,t+h_i))_{t \ge 0})_{i = 1,\ldots,M}$, $M \in \mathbb N$ to get tightness of $\U^M$. 
The above is not quite rigorous (and will be part of an upcoming paper) but should give a hint that it is important to understand the function $\F$ in order to get results on $\UM$-valued processes. \\

At this point we should note that the results presented here are generalizations and extensions of results in \cite{Grieshammer}, 
but for the sake of completeness we include all proofs needed for the results in this paper.

\section{Metric measure spaces and the Gromov weak atomic topology}\label{sec.mms}

	Here we give the definition and basic properties of (ultra-)metric measure spaces, the subspaces we are interested in and the Gromov-weak (see \cite{DGP11} and \cite{GPW09}) and Gromov-weak atomic topology. \par 

	Recall that the support of a finite Borel measure $\mu$, denoted by $\supp(\mu)$, 
	on some separable metric space $(X,d)$ is defined as the smallest closed set $C$ with $\mu(X\backslash C) = 0$. Note that  
	$\supp(\mu)$ is also given as
	\begin{equation}\label{eq_supp}
		\supp(\mu) = \{x \in X\big|\ \forall \e >0: \ \mu(B_\e(x)) > 0\},
	\end{equation}
	where $B_\e(x)$ is the open ball of radius $\e$ around $x$. 

\begin{definition}(Metric measure spaces) \label{def.mms} 
We call the triple $(X,r,\mu)$
\begin{enumerate} 
\item  a {\it  metric measure space},	short mm-space, if 
	\begin{enumerate}
			\item[(a)] $(X,r)$ is a complete separable metric space, where we assume that $X \subset \R$ (one needs this to
			avoid set theoretic pathologies).
			\item[(b)] $\mu \in \M_f(X)$, i.e. $\mu$ is a finite measure on the Borel sets  
			generated by $r$.
	\end{enumerate} 
\item	 {\it ultra-metric}, if $	r(x_1,x_2) \le r(x_1,x_3)\vee r(x_3,x_2)$	for $ \mu$-almost all $x_1,x_2,x_3$,	
\item {\it compact}, if $\supp(\mu)$ is compact,
\item {\it purely atomic}, if $\sum_{x \in X} \mu(\{x\}) = \mu(X)$, and {\it non atomic} if $\sum_{x \in X} \tilde \mu(\{x\})= 0$. 
\item {\it identifiable (by family sizes)}, if it is ultra-metric and 
\begin{equation}
 \sum_{x \in A^1_h}\mu(\bar B_h(x)) \neq \sum_{x \in A^2_h}\mu(\bar B_{h}(x)), 
\end{equation}
 for all $h > 0$ and all measurable subsets $A^1_h,A^2_h \subset \supp(\mu)$ with 
 \begin{equation}
  (A^1_h)^h := \{y \in \supp(\mu): \exists x \in A^1_h, \ r(x,y) \le h\} \neq (A^2_h)^h
 \end{equation}
 and $x,y \in A^i_h,\ x\neq y $ implies $r(x,y) > h$, $i = 1,2$.
\item {\it (non simultaneous) binary} if  $r(x_1,x_2) = r(x_3,x_4)$ implies either $x_1 = x_3$ and $x_2 = x_4$ or $x_1 = x_4$ and $x_2 = x_3$  for $ \mu$ almost all $x_1,x_2,x_3,x_4$. 
\end{enumerate}	
We say that two mm-spaces $(X,r_X,\mu_X)$ and $(Y,r_Y,\mu_Y)$ are {\normalfont equivalent} if there is a measure-preserving isometry between this spaces, i.e. a map
	$\varphi: \supp(\mu_X) \rightarrow \supp(\mu_Y)$ with $r_X(x,y) = r_Y(\varphi(x),\varphi(y))$, 
	$x,y \in \supp( \mu_X)$ and $\mu_Y = \mu_X\circ \varphi^{-1}$.
	This property defines an equivalence relation, and we denote by $[X,r,\mu]$ the equivalence 
	class of a mm-space $(X,r,\mu)$. \par 
  We define the following sets: 
	\begin{align}
	\MM &:= \left\{[X,r,\mu]: \ (X,r,\mu) \textrm{ is a metric measure space}\right\}, \\
	\UM &:= \left\{[X,r,\mu] \in \MM: \ (X,r,\mu) \textrm{ is an ultra-metric measure space}\right\},\\
	\UMatom &:= \left\{[X,r,\mu] \in \UM: \ (X,r,\mu) \textrm{ is purely atomic}\right\},\\
	\UMcont &:= \left\{[X,r,\mu] \in \UM: \ (X,r,\mu) \textrm{ is non-atomic}\right\},\\
	\UMc &:= \left\{[X,r,\mu] \in \UM: \ (X,r,\mu) \textrm{ is compact}\right\},\\
	\UMI&:= \left\{[X,r,\mu] \in \UM: \ (X,r,\mu) \textrm{ is identifiable}\right\},\label{eq.UMI}\\
	\UMB&:= \left\{[X,r,\mu] \in \UM: \ (X,r,\mu) \textrm{ is binary}\right\}\label{eq.UMB}.
	\end{align}
	We also use combinations of the above spaces such as $\UMatomc:=\UMatom\cap \UMc$ etc. 
\end{definition}
We will typically use $\mfm$ and $\mfu$ for elements of $\MM$ and $\UM$. 

\begin{remark}
Clearly, the property of being measure preserving isometric is reflexiv and transitiv. To see that it is symmetric 
one can first show that the image $\varphi(\supp(\mu_X))$ is dense in $\supp(\mu_Y)$ and then extend the inverse to a measure-preserving
isometry $\supp(\mu_Y) \to \supp(\mu_X)$. We may therefore assume w.l.o.g. that the measure-preserving isometries are surjective. 
\end{remark}

  \begin{example}\label{ex.UMI}(Identifiable elements)
 Let $(U_i)_{i \in \mathbb N}$ be independent $\mathbb R_+$-valued random variables, which are all absolutely continuous to the Lebesgue measure
 and satisfy $\sum_i U_i < \infty$ almost surely. Let $(\mathbb N,r)$ be a complete ultra-metric space, then 
 \begin{equation}
  \left [\mathbb N,r,\sum_{i \in \mathbb N}U_i \delta_i\right] \in \UMI,\qquad \text{almost surely}.
 \end{equation}
\end{example}

\begin{definition}(Distance matrix distribution) \label{def_dmd} Let $k \in \N_{\ge 2}$, $\mfm  =[X,r,\mu]\in \MM$ and set
	\begin{equation}\label{eq_distmap}
	 R^{k,(X,r)}: \left\{
	 \begin{array}{ll}
   X^{k} \rightarrow \R_+^{\binom{k}{2}} ,  \\
	 (x_{i})_{1\le i \le k} \mapsto  (r(x_{i},x_{j}))_{1\le i < j\le k}.
	 \end{array} \right.
	\end{equation} 
	We define the {\normalfont  distance matrix distribution of order $k$} by:
	\begin{equation}\label{eq_mdmd}
	\begin{split}
	 \nu^{k,\mfm} :=  (R^{k,(X,r)})_*  \mu^{\otimes k} \in \mathcal{M}_f\left(\R_+^{\binom{k}{2}}\right), 
	\end{split}
	\end{equation}
	where $ \R_+^{\binom{k}{2}}$ is equipped with the product topology.
	For $k = 1$ we define
	\begin{equation}
	 \nu^{1,\mfm} := \overline{\mfm} := \mu(X).
	\end{equation}
\end{definition}

\begin{remark} 
	Note that $\nu^{k,\mfm}$ in the above definition does not depend on the representative $(X,r,\mu)$ of $\mfm$. In particular 
	$\nu^{k,\mfm} $ is well defined for all $k \in \mathbb N$.
\end{remark}
 
\begin{definition}(Gromov-weak topology)\label{def_GS}
	Let $\mfm,\mfm_1,\mfm_2, \ldots \in \MM$. We say $\mfm_n \rightarrow \mfm$ for $n \rightarrow \infty$ in the {\it Gromov-weak
	topology}, if
	\begin{equation}
	\nu^{k, \mfm_n} \stackrel{n \rightarrow \infty}{\Longrightarrow} \nu^{k, \mfm}
	\end{equation}
	in the weak topology on $\M_f\left( \R_+^{\binom{k}{2}}\right)$ for all $k \in \mathbb N$.
\end{definition}

\begin{remark}\label{r.totalMassCont}
Since $\overline{\mfm} = \sqrt{ \nu^{2,\mfm}(\R_+)}$, we have $\mfm \mapsto \overline{\mfm}$ is continuous in the Gromov-weak topology.
\end{remark}
	
For our results it will be necessary to introduce a finer topology:

\begin{definition}(Gromov-weak atomic topology)\label{def.GromWeakAtom}  Let $\mfu,\mfu_1,\mfu_2,\ldots \in \UM$. We say 
	$\mfu_n \rightarrow 	\mfu$ for $n \rightarrow \infty$ in the {\it Gromov-weak atomic topology}, if
	$\mfu_n \rightarrow \mfu$ for $n\rightarrow \infty$ in the Gromov-weak topology and 
	\begin{itemize}
	 \item[a)] $(\nu^{2,\mfu_n})^\ast \Rightarrow (\nu^{2,\mfu})^\ast$, where $(\nu^{2,\mfu})^\ast = \sum_{h \ge 0} \nu^{2,\mfu}(\{h\})^2 \delta_h$.
	 \item[b)] $\nu^{2,\mfu_n}(\{0\}) \rightarrow \nu^{2,\mfu}(\{0\}) $. 
	\end{itemize}
\end{definition}

\begin{remark}
 This topology is related to the so called weak atomic topology on finite measures, introduced by \cite{EKatomic}, 
 where one says that a sequence $\mu_n \in \mathcal M_f(X)$, $n \in \mathbb N$ of finite Borel-measures converges to 
 a finite Borel-measure $\mu \in \mathcal M_f(X)$ in the {\it weak atomic topology}, when $\mu_n \Rightarrow \mu$ (i.e. convergence in 
 the weak topology) and $\mu_n^\ast:= \sum_{x \in X }\mu_n(\{x\})^2 \delta_x \Rightarrow \mu^\ast$. \par 
 This explains the origin of the name ``Gromov-weak atomic''. 
\end{remark}

\begin{example}(Convergence in the Gromov-weak atomic topology - Example \ref{ex1} continued)\label{ex1.1}
	Assume we are in the situation of Example \ref{ex1}, then $\mfu_n \rightarrow \mfu$ in the Gromov-weak topology.	
	Note that
	\begin{equation}
		\begin{split}
			(\nu^{2,\mfu_n})^\ast  &= 3^2 \delta_0 + 2^2 \delta_{1} + 4^2\delta_{1+ \frac{1}{n}} \\
			(\nu^{2,\mfu})^\ast &=  3^2 \delta_0  + 6^2 \delta_{1}.
		\end{split}
	\end{equation}
	Hence 
	\begin{equation}
	(\nu^{2,\mfu_n})^\ast \Rightarrow 3^2 \delta_0 + (2^2+4^2) \delta_{1}\neq (\nu^{2,\mfu})^\ast.
	\end{equation}
	This means $\mfu_n \not\rightarrow \mfu$ in the Gromov-weak atomic topology. 
\end{example}

Now recall the definition of the Prohorov distance of two finite measures $\mu_1$ and $\mu_2$ on a metric space $(E,r)$ with Borel $\sigma$-field $\B(E)$ 
\begin{equation}\label{eq_def_Pr}
\begin{split}
\dPr(\mu_1,\mu_2):= \inf \Big\{\e >0:\ &\mu_1(A) \le \mu_2(A^\e) + \e,\\
& \mu_2(A) \le \mu_1(A^\e) + \e \ \textrm{for all } A \textrm{ closed} \Big\},
\end{split}
\end{equation}
where 
\begin{equation}
A^\e:=\Big\{x \in E:\ r(x,x') < \e, \textrm{ for some } x' \in A\Big\}. 
\end{equation}

The next proposition summarizes some important facts about the Gromov-weak topology (see \cite{DGP11} and \cite{LVW} section 2.1; compare also \cite{GPW09}).

\begin{proposition}(Properties of the Gromov-weak topology)\label{prop_M_polish}
	(a) $\MM$ equipped with the Gromov-weak topology is Polish and the subspace  $\UM \subset \MM$ 
	is  closed. \par 
	(b) An example for a complete 	metric on $\MM$ (respectively $\UM$) is the {\normalfont  Gromov-Prohorov metric} 
	$\dGP$, where for two mm-spaces $[X,r_X,\mu_X]$ and $[Y,r_Y,\mu_Y]$
	\begin{equation}
		\dGP([X,r_X,\mu_X],[Y,r_Y,\mu_Y]):= \inf_{(\varphi_X,\varphi_Y,Z)} d^{(Z,r_Z)}_{\textrm{Pr}} 
		\big(\mu_X\circ \varphi_X^{-1},\mu_Y\circ \varphi_Y^{-1} \big),
	\end{equation}
	where the infimum is taken over all isometric embeddings $\varphi_X$ and $\varphi_Y$ from $\supp(\mu_X)$ and $\supp(\mu_Y)$
	into some complete separable metric space $(Z,r_Z)$ and $d^{(Z,r_Z)}_{\textrm{Pr}}$	denotes the Prohorov distance on $\mathcal M_f(Z)$.
\end{proposition}

We close this section with some properties of the Gromov-weak atomic topology: 

\begin{theorem}(Properties of the Gromov-weak atomic topology)\label{thm_polish}
 If we equip $\UM$ with the Gromov-weak atomic topology then the following holds (recall \eqref{eq.UMI}): 
\begin{itemize}
\item[(a)] $\UM$  is a Polish space, 
\item[(b)] $\UMI \subset \UM$ is dense. 
\item[(c)] Let $\UMc$ be equipped with the subspace topology, then $\UMB \subset \UMc$ is closed.
\end{itemize}
\end{theorem}

\begin{remark}\label{rem.Borel}
(1) As in \cite{EKatomic} (see the discussion after (2.3)), the Borel sets generated by the Gromov-weak topology coincide with the Borel-sets generated by the 
Gromov-weak atomic topology.\par 
(2)  $\UMc \subset \UM$ is measurable in the Gromov-weak topology and therefore measurable in the Gromov-weak atomic topology (see Remark 2.8 and Corollary 3.6 in \cite{ALW}). 
\end{remark}

\section{Family size decomposition of ultra-metric measure spaces}\label{sec.decomp}

We will now introduce the function $\F$ that gives the size of the different families of an ultra-metric measure space $\mfu$.

\subsection{Definitions}

We start with the following Lemma, that gives us the existence of an ``almost surely'' disjoint decomposition  of an ultra-metric measure space into  closed balls. 

\begin{lemma}\label{lem_representatives} 
	Let $0< h$, $\mfu = [X,r,\mu] \in \UM$ and  $\bar B_h (x)$ be the closed ball of radius 
	$\le h$ around $x \in X$. Then there is a $\num(h) \in \N \cup \{\infty\}$ and a family 
	$\{\rep_i^\h: \ i \in \{1,2,\ldots,\num(h)\}\}$ of elements of $\supp(\mu)$ with
	\begin{equation}
			\mu\big(\bar B(\rep_i^h,h) \cap \bar B(\rep_j^h,h)\big) = 0,
	\end{equation}
	for $i \neq j$ and
	\begin{equation}
			\mu(X) = \sum_{i =1}^{\num(h)} \mu\big(\bar B(\rep_i^h,h)\big).
	\end{equation}
	Moreover, if $0< \delta \le h$, then  there is a partition $\{I_i\}_{i \in 1,\ldots,n(h)}$ of $\{1,\ldots,\num(\delta)\}$ such that
	\begin{equation}
		\mu(\bar B(\rep_i^h,h))= \sum_{j \in I_i}\mu(\bar B(\rep_j^\delta,\delta)),\qquad \forall i = 1,\ldots,\num(h).
	\end{equation}
\end{lemma}

\begin{remark}\label{rem_pos} 
	(i) By the definition of the support we get $ \mu(\bar B(\rep_i^h,h)) >0$ for all $i \in \{1,\ldots,\num(h)\}$. \par 
	(ii) The analogue of Lemma \ref{lem_representatives}  holds if we replace $\le h$ by $< h$. 
\end{remark}

\begin{remark}\label{rem.reconstruction}
Another important observation is the following: Given a finite ultra-metric space $(\{1,\ldots,\num(0)\},r) = (X,r)$, then the path of partitions 
$h \mapsto (\{I_i^h\}_{i = 1,\ldots,\num(h)})$ of $\{1,\ldots,\num(0)\}$ contains all information of the metric $r$, i.e. the function that maps an ultra-metric $r$ on $X$ to the path of partitions is an injection and given an element $(h \mapsto \pi^h)$ contained in the range of this map, the corresponding metric $r$ can be reconstructed by 
\begin{equation}
r(k,l) := \inf\{h > 0: \ k,l \in \pi^h_i \text{ for some } i = 1,\ldots,\num(h)\}, 
\end{equation}  
for $k,l \in \{1,\ldots,\num(0)\}$.
\end{remark}

\noindent Let $C \ge 0$ and set 
\begin{align}
\mathcal S^\downarrow_C&:=\left\{(x_1,x_2,\ldots) \in [0,\infty)^\N:\ \sum_{i \in \N} x_i \le C,\ x_1 \ge x_2 \ge \ldots \right\},\\
\mathcal S^\downarrow&:=\left\{(x_1,x_2,\ldots) \in [0,\infty)^\N:\ 
\sum_{i \in \N} x_i < \infty,\ x_1 \ge x_2 \ge \ldots \right\}.
\end{align}
\noindent We consider the following two distances on $\mathcal S^\downarrow_C$ and $\mathcal S^\downarrow$:
\begin{equation}\label{eq_metric_S}
\dSeq(x,y) = \sum_{i = 1}^\infty |x_i - y_i| = \norm{x - y}_1
\end{equation}
and
\begin{equation} 
\dMax(x,y) = \max_{i \in \mathbb N} |x_i - y_i|.
\end{equation}
We note that $\mathcal S^\downarrow_C$ and $\mathcal S^\downarrow$ are typically equipped with the $\ell^1$-distance, $\dSeq$.

\begin{definition}(Definition of $\f$) \label{defin_basic}
 Let $\mfu \in \UM$.  We define the map $\f(\mfu,\cdot): (0,\infty) \rightarrow \mathcal S^\downarrow$,
				\begin{equation}
						\f(\mfu,h) = (a_1(h),a_2(h),\ldots),
				\end{equation}
				where the $a_k(h)$ are given by 
				\begin{equation}
						\begin{split}\label{eq_ord}
								a_k(h) &= \max\left\{c \ge 0:\ \sum_{i=1}^{\num(h)}  \mathds{1}(\mu(\bar B(\rep_i^h,h)) \ge c) \ge k\right\}, 
								\quad k = 1,2,\ldots,\num(h), \\
								a_k(\h) &= 0,\quad \textrm{ for } k > \num(h). 
						\end{split}
				\end{equation}
				Note that $a_k(h) \ge a_{k+1}(h)$ is the non-increasing reordering of the sequence $(\mu(\bar B(\rep_i^h,h)))_{i = 1,\ldots,\num(h)}$.
\end{definition}

\begin{remark}\label{rem_ind} ${}$
	\begin{itemize}
		\item[(i)] Let $(X,r_X,\mu_X)$ and $(Y,r_Y,\mu_Y)$ be two equivalent ultra-metric measure spaces and let 
				$\varphi: \supp(\mu_X) \rightarrow \supp(\mu_Y)$ be a measure preserving isometry. 
				Then $\{\rep_i^h: \ i \in \{1,2,\ldots,\num(h)\}\} \subset \supp(\mu_X)$  satisfies the conditions in Lemma 
				\ref{lem_representatives} if and only if $\{\varphi(\rep_i^h): \ i \in \{1,2,\ldots,\num(h)\}\} \subset \supp(\mu_X)$ 
				satisfies the conditions.
		\item[(ii)] If $x \in \supp(\mu_X)$ and $h >0$, then  there is exactly one $i \in \{1,\ldots,\num(h)\}$ with
				\begin{equation}
						 \mu_X(\bar B(\rep_i^h,h)) =\mu_X(\bar B(\rep_i^h,h) \cap \bar B(x,h))  =\mu_X(\bar B(x,h)).
				\end{equation}
	\end{itemize}
	As a consequence, the definition of $\f$ does not depend on the representatives. 
\end{remark}

Note that  the domain of $\f(\mfu,\cdot)$ is $(0,\infty)$. In some cases it is also possible to add $0$ to the domain and we close this section with the following remark:  
\begin{remark}\label{rem_extentzero}
	In the case, where $\mfu \in \UMatom$ is purely atomic  we can extend the function $\f(\mfu,\cdot)$ to a function 
	$\hat \f(\mfu,\cdot) :[0,\infty) \rightarrow \mathcal S^\downarrow$. 
\end{remark}

\subsection{Results} 

We start with the following definition:

\begin{definition}(Definition of $\F$)\label{def_calF}
	We define
	\begin{equation}
	\F:\UM \rightarrow (\mathcal S^{\downarrow})^{(0,\infty)}, \quad \mfu \mapsto \f(\mfu,\cdot).
	\end{equation}
\end{definition}

The first observation is, that $\F$  maps ultra-metric measure spaces to cadlag (i.e. right continuous with left limits) functions:

\begin{lemma}\label{l.F.cadlag}
$\domF:= \F(\UM) \subset D((0,\infty),\mathcal S^\downarrow)$, where $\mathcal S^\downarrow$ is equipped with $\dSeq$.
\end{lemma}

In the following 

\begin{itemize}
\item[] $D((0,\infty),\mathcal S^\downarrow)$ is always equipped with the Skorohod topology, given in Appendix \ref{sec.Skorohod}.
\end{itemize}

Now the question is whether $\F$ is continuous when $\UM$ is equipped with the Gromov-weak topology. 

\begin{example} Assume we are in the situation of Example \ref{ex1}.  Observe that if we  take for example 
	$t_n = 1+\frac{1}{n} \rightarrow 1$ then 
	\begin{equation}
	\begin{split}
	\f(\mfu_n,t_n) &\equiv \left(2,1,0,\ldots\right) \\
	&\not\in \left\{\left(1,1,1,0\ldots\right),\left(3,0,0,\ldots\right) \right\} = 
	\left\{\f(\mfu,1),\f(\mfu,1-)\right\},
	\end{split}
	\end{equation}
	i.e. $\f(\mfu_n,\cdot) \not\rightarrow \f(\mfu,\cdot)$ in the Skorohod topology (see Proposition 3.6.5 in \cite{EK86}). 
\end{example}

	In other words we can not expect $\F$ to be continuous, when $\UM$ is equipped with the Gromov-weak topology. But as we have seen in 
	Example \ref{ex1}, the sequence $\mfu_n$ does not converge in the Gromov-weak atomic topology and in fact, this is the reason why we introduced this new topology. \\
	
	Recall that  a function $f: X \to Y$ between two topological spaces is called \emph{perfect}, 
if it is continuous, surjective, closed (i.e. maps closed sets to closed sets) and $f^{-1}(\{y\})$ is compact in $X$ for all $y \in Y$. We remark the following: 

\begin{remark}\label{r.perfect} If $X$ is a topological space and $Y$ is a compactly generated Hausdorff space (for example a metric space) and $f: X \to Y$ is surjective, then the following is equivalent (see for example \cite{palais1970proper}): 
\begin{itemize}
\item[(i)] $f$ is perfect,
\item[(ii)] $f$ is continuous and proper, i.e. $f^{-1}(K)$ is compact in $X$ for all compact sets $K \subset Y$. 
\end{itemize}
Note that a perfect map is also a quotient map, i.e. surjective and $f^{-1}(U)$ is open in $X$ iff $U$ is open in $Y$. 
\end{remark}

\begin{theorem}(Properties of $\F$)\label{thm.perfect}  
Let $\UM$ be equipped with the Gromov-weak atomic topology, then $\F: \UM \rightarrow \domF$ has the following properties: 
\begin{itemize}
\item[i)] $\F$ is perfect. 
\item[ii)] The restriction $\F|_{\UMI}$ of $\F$ to $\UMI$ (see Definition \ref{def.mms}) is a homeomorphism onto its image. 
\end{itemize}
\end{theorem}

Recall that a collection of cadlag process $\{X^n:\ n\in\mathbb N\}$ with values in some Polish space $E$ satisfies a compact containment condition if for all $\eps > 0$ and $T > 0$ there is a compact set $K \subset E$ such that 
\begin{equation}
\inf_{n \in \mathbb N} P(X^n(t) \in K\ \forall t \in [0,T]) \ge 1-\eps. 
\end{equation}

\begin{corollary}\label{c.perfect}
Let $\mathcal U_n$ be a sequence in $\UM$ and $\UM$ be equipped with the Gromov-weak atomic topology. Then 
\begin{itemize}
\item[(i)] $(\mathcal L(\U_n))_{n \in \mathbb N}$ is tight if and only if $(\mathcal L(\F(\U_n)))_{n \in \mathbb N}$ is tight.
\item[(ii)] $\U_n \Rightarrow \U$ for some $\UM$-valued random variable $\U$ implies $\F(\U_n)\Rightarrow \F(\U)$. 
\end{itemize}
Moreover, the map $D([0,\infty),\UM) \rightarrow  D([0,\infty),\F(\UM)), (\mfu_t)_{t \ge 0} \mapsto (\F(\mfu_t))_{t \ge 0}$ is continuous and a collection of $\UM$-valued cadlag processes $\{(\U_t^n)_{t \ge 0}:\ n\in \mathbb N\}$ satisfies a compact containment condition if and only if $\{(\F(\U_t^n))_{t \ge 0}:\ n\in \mathbb N\}$ satisfies a compact containment condition. 
\end{corollary}

\begin{proof}
This is a direct consequence of Theorem \ref{thm.perfect}, the continuous mapping theorem, Prohorov's theorem and the fact that continuous images of compact sets are compact.  See Problem 3.13 in \cite{EK86}. 
\end{proof}

Another interesting observation is, that even though, the above result is related to the Gromov-weak atomic topology, we also get a 
result for the Gromov-weak topology: 

\begin{proposition}\label{p.tightness.GW}
Let $\U^n$, $n = 1,2,\ldots$ be a sequence of $\UM$-valued random variables and let $\UM$ be equipped with the Gromov-weak topology. 
Assume that for all $\delta > 0$ and all $\eps > 0$ 
\begin{itemize}
\item[(i)] there is a compact set $\Gamma \subset \mathcal S^\downarrow$ such that 
\begin{equation}
\limsup_{n \rightarrow \infty} P\left(\f(\U_n,\delta) \in \Gamma^c\right)  \le \eps,
\end{equation}
\item[(ii)] there is an $H \ge  0$ such that 
\begin{equation}
\limsup_{n \rightarrow \infty} P\left(\left(\sum_{i = 1}^\infty \f(\U^n,H)_i\right)^2-\sum_{i = 1}^\infty \f(\U^n,H)^2_i \ge \eps\right) \le \eps
\end{equation}
\end{itemize}
and that the total mass $(\nu^{1,\U^n})_{n \in \mathbb N}$ is tight. Then, $(\U^n)_{n \in \mathbb N}$ is tight. 
\end{proposition}

\begin{remark}\label{rem.compact.S}
Note that $\mathcal S^\downarrow_C$ equipped with $\dMax$ is a compact space (this follows analogue to Proposition 2.1. in \cite{B}). 
It is not hard to see that $\Gamma\subset \mathcal S^\downarrow_C$, equipped with $\dSeq$, is compact, if for all $\eps > 0$ there is a 
$M \in \mathbb N$ such that   
\begin{equation}
\sup_{f \in \Gamma} \sum_{i \ge M} f_i \le \eps.
\end{equation}
We will discuss this property in more detail in Section \ref{sec.relative.compactness.F}. 
\end{remark}

Even though the function $\F$ is not injective on the whole space it is at least injective on a dense subset (see Theorem \ref{thm_polish}).

\begin{remark}\label{r.hom.ext.1}
By Lavrentiev's Theorem (see Section 35.II in \cite{kuratowski2014topology}) there are two $G_\delta$ sets $\UMI \subset \UMI^\ast \subset \UM$
and $\F(\UMI) \subset \UMI^\ast_{\F} \subset \F(\UM)$ and a homeomorphism $\F^\ast:\UMI^\ast \to \UMI^\ast_{\F} $ extending $\F$. 
In addition $\UMI^\ast $ is dense in $\UM$ since $\UMI$ is dense in $\UM$.
\end{remark}

Next, we give a criterion when a $\UM$-valued random variable takes values in $\UMI$:
\begin{proposition}\label{p.UMI.as}
Assume that $\mathcal U$ is an $\UM$-valued random variable. Let $N_h \in \mathbb N \cup \{\infty\}$ be the number of non-zero entries of $\f(\U,h)$. 
If $\mathcal L(\f(\U,h)) \ll \lambda^{\otimes N_h}$ conditioned on $N_h$ for all $h > 0$, where $\lambda$ denotes the Lebesgue measure, then $\mathcal U \in \UMI$ almost surely.  
\end{proposition}

We close this section with a result that is even stronger than the above result, when one considers the subspace 
$\UMB\cap \UMI$ (see Definition \ref{def.mms}):

\begin{theorem}\label{thm.subspace.homeomorphism}(Properties of $\UMB\cap\UMI$)
Let $\UMB\cap\UMI$ be equipped with the Gromov-weak atomic topology, 
then we have $\NU:\UMB\cap\UMI \to \NU(\UMB\cap\UMI) \subset \mathcal M_f(\mathbb R_+)$, $\mfu \mapsto \nu^{2,\mfu}$ is a homeomorphism. 
\end{theorem}

\begin{remark}\label{r.hom.ext.2}
As in Remark \ref{r.hom.ext.1} we can extend the homeomorphism to $G_\delta$ subsets. 
\end{remark}

\section{Application to the tree-valued Fleming-Viot process}\label{sec.app.FV}

In this section we give a short introduction to tree-valued Fleming-Viot processes and show that these processes live in the 
subspace $\UMB\cap\UMI$ (see Theorem \ref{thm.subspace.homeomorphism}). For simplicity, we will only introduce the  
neutral model and refer to \cite{DGP12} for the general case.  \par 
In section \ref{sec.neutralMM} we define the neutral tree-valued Moran model of a given size $N$ (the population size).  This model was defined by  \cite{GPW13} and extended by \cite{DGP12}  to include selection and mutation. 
In section \ref{sec.FV.result} we consider the large population limit (i.e. $N \rightarrow \infty$) of the tree-valued Moran models, the so called tree-valued Fleming-Viot process,  and give our main result for this process.

\subsection{Definition of the neutral model}\label{sec.neutralMM}

We want to describe the genealogy of a population, consisting of  $N \in \mathbb N$ individuals, that evolves according to the following dynamic:

\begin{itemize}
\item[] {\it Resampling:} Every pair $i \neq j$ of individuals is replaced with rate one. 
If such an event occurs, $i$ is replaced by an offspring of $j$ with probability $\frac{1}{2}$, or $j$ is replaced by an
offspring of $i$ with probability $\frac{1}{2}$. 
\end{itemize}

In order to describe the evolution of this process formally, let $I_N:=[N]:= \{1,\ldots,N\},\ N \in \mathbb N$  and 
\begin{equation}\label{ppp}
\left\{\eta^{i,j}:\ i,j \in I_N,\ i \not=j\right\}
\end{equation}
be a realization of a family of independent rate $1$ Poisson point processes. 
\begin{itemize}
\item[] For $i,i' \in I_N$, $0\le h< t < \infty$ we say that there is a {\it path} from $(i,h)$ to $(i',t)$ if there is an  $n \in \mathbb N$, 
$h \le t_1 < t_2 < \cdots < t_n \le t$ and
$j_1,\ldots,j_n \in I_N$ such that for all $k \in \{1,\ldots,n+1\}$ ($j_0:= i, j_{n+1}:=i'$)
$\eta^{j_{k-1},j_k}\{t_k\} = 1$, $\eta^{x,j_{k-1}}((t_{k-1},t_k))= 0$ for all $x \in I_N$.
\end{itemize}

\begin{figure}[ht]
\begin{center}
\begin{tikzpicture}[scale = 0.68]
\draw [->] (2,-4) -- (2,5);
\draw [line width=1.pt] (4,-3)-- (4,4);
\draw (6,-3)-- (6,4);
\draw (8,-3)-- (8,4);
\draw (10,-3)-- (10,4);
\draw [->,line width=1.pt] (4,-2) -- (6,-2);
\draw [->,line width=1.pt] (10,-1) -- (6,-1);
\draw [->,line width=1.pt] (6,2) -- (4,2);
\draw [->,line width=1.pt] (6,3.5) -- (8,3.5);
\draw [line width=1.pt] (4,3)-- (4,4);
\draw [line width=1.pt] (6,-3)-- (6,-2);
\draw [line width=1.pt] (6,-2)-- (6,-1);
\draw [line width=1.pt] (6,-1)-- (6,1.5);
\draw [line width=1.pt] (6,1.5)-- (6,2);
\draw (6,1.5)-- (6,2);
\draw [line width=1.pt] (6,1.5)-- (6,4);
\draw [line width=1.pt] (10,-3)-- (10,4);
\draw [line width=1.pt] (8,-3)-- (8,-0.5);
\draw [line width=1.pt] (8,-0.5)-- (8,3.5);
\draw [line width=1.pt] (8,3.5)-- (8,4);
\draw (4,-3) node[anchor=north] {$1$};
\draw (6,-3) node[anchor=north] {$2$};
\draw (8,-3) node[anchor=north] {$3$};
\draw (10,-3) node[anchor=north] {$4$};
\draw (2,4) node[anchor=east] {$t$};
\draw (2,-1.5) node[anchor=east] {$h$};
\draw (2,5) node[anchor=east] {$\text{time}$};
\draw [line width=1.pt] (15,4)-- (15,3.5);
\draw [line width=1.pt] (15,3.5)-- (17,3.5);
\draw [line width=1.pt] (17,3.5)-- (17,4);
\draw [line width=1.pt] (13,4)-- (13,2);
\draw [line width=1.pt] (16,3.5)-- (16,2);
\draw [line width=1.pt] (13,2)-- (16,2);
\draw [line width=1.pt] (19,4)-- (19,-1);
\draw [line width=1.pt] (19,-1)-- (14.5,-1);
\draw [line width=1.pt] (14.5,-1)-- (14.5,2);
\begin{scriptsize}
\draw (2,4)-- ++(-2.5pt,-2.5pt) -- ++(5pt,5pt) ++(-5pt,0) -- ++(5pt,-5pt);
\draw (2,-1.5)-- ++(-2.5pt,-2.5pt) -- ++(5pt,5pt) ++(-5pt,0) -- ++(5pt,-5pt);
\end{scriptsize}
\end{tikzpicture}
\caption{\label{fig.TVMM} \footnotesize On the left side we see the graphical construction of the Moran model; $\rightarrow$ indicates a resampling event. On the right side we see the genealogical tree of the population at time $t$. 
In this case the ancestor of all individuals at time $h$ would be individual $4$, i.e.  $A_h(i,t) = 4$ for all $i = 1,\ldots,4$}
\end{center}
\end{figure}
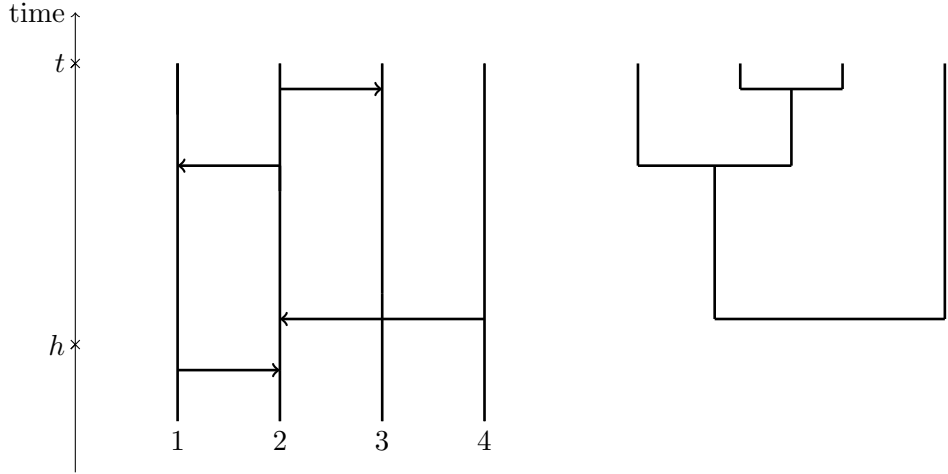

\noindent Note that for all $i \in I_N$ and $0 \le h \le t$ there exists an unique element
\begin{equation}\label{eq.ancestors}
A_h(i,t) \in I_N
\end{equation}
with the property that there is a path from $(A_h(i,t),h)$ to $(i,t)$. We call $A_h(i,t)$ the {\it ancestor of  $(i,t)$ at time $h$} (see Figure \ref{fig.TVMM}). \par

Let $r_0$ be an ultra-metric on $I_N$ and  $i,j \in I_N$. Then we define the following (pseudo) ultra-metric on $I_N$:
{\small
\begin{equation}\label{eq.ultrametric}
\begin{split}
r_t&(i,j):=\left\{ \begin{array}{ll}
t-\sup\{h \in [0,t]:A_h(i,t)=A_h(j,t)\},&\textrm{ if } A_0(i,t) = A_0(j,t),\\[0.2cm]
t + r_0(A_0(i,t), A_0(j,t)) ,& \textrm{ if } A_0(i,t) \not= A_0(j,t).
\end{array}\right.
\end{split}
\end{equation}
}

Since $r_t$, is only a pseudo-metric, we  consider the following equivalence relation $\approx_t$  on $I_N$: 
 $x \approx_t y \Leftrightarrow r_t(x,y) = 0$.
We denote by $\tilde I_N^t:= I_N\! /\!\!\approx_t$ the set of equivalence classes and note that we can find a set of representatives $\bar I_N^t$
such that $\bar I_N^t \rightarrow \tilde I_N^t,\  x \to [x]_{\approx_t}$ is a bijection. \par 
Let  $\mu^N\in \mathcal M_1(I_N)$  be the uniform distribution on $I_N$, i.e. 
\begin{equation}
\mu^N = \frac{1}{N} \sum_{k \in I_N} \delta_{k}
\end{equation}
and define
\begin{align}
\bar r_t(\bar  i,\bar  j) =  r_t(\bar i,\bar j),\quad \bar \mu^N_t(\{\bar i\}\times \cdot) &=\mu^N_t( [\bar i]_{\approx_t} \times \cdot),\qquad  \bar i,\bar j \in \bar I_N^t.
\end{align} 
Then the {\it tree-valued Moran model} of size $N$ is defined as
\begin{equation}
\mathcal U_t^{N} := [\bar I_N^t,\bar r_t,\bar \mu^N_t], 
\end{equation}

\subsection{Results for the tree-valued Fleming-Viot process}\label{sec.FV.result}

Assume that $\L(\U_0^{N}) \Rightarrow \mu \in \mathcal M_1(\UM)$, where $\UM$ is equipped with the Gromov-weak topology. Then 

\begin{equation}
(\U^{N}_t)_{t \ge 0} \stackrel{N \rightarrow \infty}{\Rightarrow} (\U_t)_{t \ge 0}
\end{equation}
weakly in the Skorohod topology on $D([0,\infty),\UM)$, where $\L(\U_0) = \mu$ and $(\U_t)_{t \ge 0}$ is the solution of a well-posed martingale problem 
(see Theorem 2 in \cite{GPW13}). We call the process $\U = (\U_t)_{t \ge 0}$  {\it tree-valued Fleming-Viot process}. \\

\begin{proposition}\label{p.convMM} (Convergence of the tree-valued Moran models)
Let $\UM$ be equipped with the Gromov-weak atomic topology and let $(V_i)_{i \in \mathbb N}$ be a sequence of independent $[0,1]$-uniformly distributed random variables.
We assume that $\U^N_0 = [[0,1],r_0,\mu^N]$, where $([0,1],r_0)$ is a compact binary ultra-metric space, and $\mu^N = \frac{1}{N} \sum_{i = 1}^N \delta_{V_i}$. 
Then 
\begin{equation}
 \U^{N}_t \Rightarrow \U_t \qquad \text{ for all } t \ge 0. 
\end{equation}
\end{proposition}

\begin{remark}
Even though we choose a special initial condition, the proof for general initial conditions should be 
similar but more technical (one needs to use Lemma 5.8 in \cite{GPW09} for example). We also note that the initial condition  does not really matter when one wants to study genealogical properties that 
are generated by an evolving population. 
\end{remark}

We are now ready for our main result: 
\begin{theorem}\label{thm.FV.neutral}(State space of tree-valued FV-processes)
Recall Remark \ref{rem.Borel} and assume that $P(\U_0 \in \UMB\cap\UMc) = 1$, then 
\begin{equation}
P(\U_t  \in \UMB \cap \UMI\cap\UMc) = 1, \qquad \forall t > 0. 
\end{equation}
\end{theorem}

\begin{remark}
Even though, it is not hard to see that $\UMB \cap \UMI$ is measurable we can also apply 
Remark \ref{r.hom.ext.2} and replace $\UMB \cap \UMI$ in the above theorem by a suitable $G_\delta$-set.
\end{remark}

Since we did not define the model with selection we need to refer all interested readers to \cite{DGP12}. 
But, as a direct consequence of the Girsanov transform - Theorem 2 in this paper, one can prove the following.

\begin{corollary}\label{cor.FV.selection}
If we denote by $\U^\alpha$ the tree-valued Fleming-Viot process with mutation and selection parameter $\alpha \ge 0$, defined in \cite{DGP12}, 
with $P(\U_0^\alpha \in \UMB\cap\UMc) = 1$, then 
\begin{equation}
P(\U_t^\alpha  \in \UMB \cap \UMI\cap \UMc) = 1, \qquad \forall t > 0. 
\end{equation}
\end{corollary}

\section{Preparations for the proofs}

We start with some preparations needed for the proofs of our results. 
In section \ref{sec.bounds} we prove some bounds for the Gromov-Prohorov metric 
and in section \ref{sec.concatenation} we introduce the notion of concatenation of trees, which will be useful in order to prove the continuity of $\F$. 

\subsection{Bounds for the Gromov-Prohorov metric and the function \texorpdfstring{$\Cut$}{Phi}}\label{sec.bounds}

We start with the following observation. 
 
\begin{remark}\label{rem_Phi}
Let $(X,r,\mu)$ and $(\tilde X,\tilde r,\tilde \mu)$ be two equivalent ultra-metric measure spaces. 
If we denote by $\{\rep_i^h:\ i = 1,\ldots,\num(h)\}$
and  $\{\tilde \rep_i^h:\ i = 1,\ldots,\tilde \num(h)\}$ two families of representatives in the sense of 
Lemma \ref{lem_representatives}, then it is not hard to see (see also Remark \ref{rem_ind}) that 
\begin{equation}
\begin{split}
&\left[\{\rep_i^h:\ i \in \{1,\ldots,\num(h)\}\},\ r, \ \sum_{i \in \{1,\ldots,\num(h)\}} 
\mu(\bar B^{r}(\rep_i^h,h)) \delta_{\rep_i^h}\right] \\
&{}\hspace{2cm} = \left[\{\tilde \rep_i^h:\ i \in \{1,\ldots,\tilde \num(h)\}\},\ \tilde r, 
\ \sum_{i \in \{1,\ldots,\tilde \num(h)\}} \tilde  \mu(\bar B^{\tilde r}(\tilde \rep_i^h,h)) 
\delta_{\tilde \rep_i^h}\right]
\end{split}
\end{equation}
and it is possible to define for $h > 0$
\begin{equation}
\CutTwo_h(\mfu) = \left[\{\rep_i^h:\ i \in \{1,\ldots,\num(h)\}\},\ r, \ 
\mu_h\right]
\end{equation}
and 
\begin{equation}
\Cut_h(\mfu) = \left[\{\rep_i^h:\ i \in \{1,\ldots,\num(h)\}\},\ r- h \cdot 1(\rep_i^h \neq \rep_j^h), \ \mu_h\right],
\end{equation}
where
\begin{equation}
\mu_h:=\sum_{i \in \{1,\ldots,\num(h)\}} \mu\left(\bar B(\rep_i^h,h)\right) \delta_{\rep_i^h}.
\end{equation}
\end{remark}

These functions will appear in several proofs. The reason is the following Lemma: 
 
\begin{lemma}\label{lem_restr}
Let $0<\h$ and $\mfu =  [X,r,\mu]\in \UM$.
\begin{itemize}
\item[(i)] If $A \subset X$ is measurable, and $\mu_A(\cdot):= \mu(\cdot \cap A)$ then 
\begin{equation}
\dGP ([A,r,\mu_A], [X,r,\mu]) \le \mu(X\backslash A).
\end{equation}
\item[(ii)] If $\mfu' = [X,r,\mu'] \in \UM$, then 
\begin{equation}
\dGP(\mfu,\mfu') \le \dPr(\mu,\mu'),
\end{equation}
where the Prohorov distance is taken on the set of Borel-measures on $X$ (see \eqref{eq_def_Pr}). 
\item[(iii)]  Let $\Cut_h$ and $\CutTwo_h$ be the functions from Remark \ref{rem_Phi}. Then
\begin{equation}
\dGP(\mfu, \CutTwo_h(\mfu)) \le h,\qquad \dGP(\Cut_h(\mfu),\CutTwo_h(\mfu)) \le h.
\end{equation}
\item[(iv)] The functions $h \mapsto \Cut_h(\mfu)$ and $h \mapsto \CutTwo_h(\mfu)$ as functions from $(0,\infty) \to \UM$ are both cadlag.
\end{itemize}
\end{lemma}

\begin{proof}
(i) Note that  the identity $id:X \rightarrow X$ is an isometric embedding from $A$ to $X$. Using the definition of the Gromov-Prohorov metric from Proposition \ref{prop_M_polish}, it is enough to bound (note that $\mu_A \le \mu$):
\begin{equation}
\begin{split}
d_{Pr}(\mu_A,\mu) = \inf\big\{\epsilon > 0: \mu (B) \le &\mu_A(B^\e) + \e, \\
&\forall B \subset X \text{ Borel-measurable}\big\},
\end{split}
\end{equation}
where 
\begin{equation}
B^\e = \{x\in X:\ \exists x' \in B,\ r(x,x') < \e\}. 
\end{equation}
Note that if $\mu(X\backslash A) = 0$ then  $d_{Pr}(\mu_A,\mu) =0$ and if $\mu(X\backslash A) > 0$ we can take $\epsilon = \mu(X\backslash A)$ and the result follows. \\

\noindent (ii) As in (i) one can use the identity as isometric embedding. \\

\noindent (iii) We use the notation of Remark \ref{rem_Phi} and note that $id$ is an isometric embedding of $\{ \rep_i^\h, \ i\in \N\}$ in $X$. Define the  measure $\bar \mu$ on $X \times X$ by
\begin{equation}
\bar \mu ( A_1  \times A_2 ):= \sum_{i \in \N}  \mu(A_1\cap \bar B(\rep_i^\h,\h)) \delta_{\rep_i^\h}(A_2). 
\end{equation}
for all measurable sets $A_1,A_2 \subset X$ and observe that $\bar \mu$ is a coupling of $\mu$ and $\mu_h$. Since $\mu_h(\{ \rep_i^\h , \ i\in \N\}) = \mu(X)$ and by the definition of the Gromov-Prohorov metric from Proposition \ref{prop_M_polish} together with Theorem 3.1.2 in \cite{EK86} (with the obvious extension to couplings of finite measures with the same mass), we get
\begin{equation}
\begin{split}
\dGP(\mfu,&\CutTwo_\h(\mfu)) \le\\
&\inf_{\nu}\inf\Big\{ \epsilon > 0:\ \nu \big(\{(x,x') \in X\times X: r(x,x')\ge  \e\}\big) \le \e\Big\},
\end{split}
\end{equation}
where the infimum is taken over all couplings $\nu$ of $\mu$ and $\mu_h$. It follows that
\begin{equation}
\begin{split}
\dGP(\mfu,\CutTwo_\h(\mfu))\le \inf\{ \epsilon > 0:\ \bar \mu (\{(x,x') \in X \times X: r(x,x')\ge  \e\}) \le \e\} 
\end{split}
\end{equation}
and if we choose $\e >\h$ then 
\begin{equation}
\begin{split}
\bar \mu &\left(\big\{(x,x') \in X \times X: r(x,x') \ge  \e\big\}\right) \\
&\quad \le \sum_{i,j \in \N, \ i\not=j} \mu(\bar B(\rep_i^\h,\h)\cap \bar B(\rep_i^\h,\h)) \delta_{\rep_i^\h}(\bar B(\rep_j^\h,\h))= 0.
\end{split}
\end{equation}
For the second part, we use the same argument as in section 3 in \cite{Loehr13}: Let $Y:=\{\rep_i^h:\ i \in \{1,\ldots,\num(h)\}$, $r^1 = r$, $r^2 = r - h 1(x\neq y)$ and $\mu^1 = \mu^2 = \mu_h$.
We denote by $Y \uplus Y$ the disjoint union of $Y$ and $Y$ and let $\varphi_i: Y\to Y \uplus Y$ be the canonical embeddings, $i = 1,2$. Define the metric $d$ on $Y \uplus Y$ by 
\begin{align}
 d( \varphi_1(x),\varphi_1(y)) &= r^1(x,y) , \\
 d( \varphi_2(x),\varphi_2(y)) &= r^2(x,y) ,  \\
 d( \varphi_1(x),\varphi_2(y)) &= \inf_{z \in Y} (r^1(x,z)+r^2(y,z))+ h, 
\end{align}
where $x,y \in Y$. Then, as in \cite{Loehr13} it is easy to see that this is a metric on $Y\uplus Y$ that extends the metrics $r^1 $ and $r^2$ (i.e. $\varphi_i$ is an isometry for $i = 1,2$) and we have 
\begin{equation}
\varphi_2(\varphi_1^{-1}(F)) \subset F^{h_0}:= \{x \in Y \uplus Y :\ \exists x' \in F \text{ s.t. } 
d(x,x') <h_0\}, 
\end{equation}
for all $h_0 > h$. Since $\mu^1 = \mu^2$ this gives: 
\begin{equation}
\mu^1 \circ \varphi_1^{-1}(F) = \mu^2 \circ \varphi_1^{-1}(F) \le \mu^2 \circ \varphi_2^{-1} (\varphi_2(\varphi_1^{-1}(F)) 
\le \mu^2 \circ \varphi_2^{-1} (F^{h_0}) + h_0, 
\end{equation}
for all $h_0 > h$ and the result follows. \\

\noindent (iv) A similar argument as in (iii) shows that $\CutTwo_{h'}(\mfu) \rightarrow\CutTwo_{h}(\mfu)$ for $h' \downarrow h$ and by definition we have 
$\Cut_{h + \delta} (\mfu) = \Cut_\delta(\Cut_h(\mfu))$ and hence by (iii) $\Cut_{h'}(\mfu) \rightarrow\Cut_{h}(\mfu)$ for $h' \downarrow h$. 
This shows the right continuity. For the existence of the left limits set  
\begin{align}
\mfu_h &:= \left[\{\rep_i^h:\ i \in \{1,\ldots,\num(h)\}\},\ r- h \cdot 1(\tau_i \neq \tau_j), \ \mu_h^\circ\right], \\ 
\hat \mfu_h &:= \left[\{\rep_i^h:\ i \in \{1,\ldots,\num(h)\}\},\ r, \ \mu_h^\circ\right], 
\end{align}
where 
\begin{equation}
\mu_h^\circ(A):= \sum_{i\in \{1,\ldots,\num(h)\}} \mu\left(B(\rep_i^h,h)\right) \delta_{\rep_i^h}(A)
\end{equation} 
is given in terms of open balls with radius $< h$ (instead of $\le h$ - see  Remark \ref{rem_pos} (ii)). 
By a similar argument as in (iii) together with the fact that 
\begin{equation}
 \lim_{h'\uparrow h} \mu\left(B(\rep_i^h,h) \setminus \bar B(\rep_i^h,h')\right) = 0,
\end{equation}
we get 
\begin{equation}
\dGP(\Cut_{h'}(\mfu),\mfu_h) \vee \dGP(\CutTwo_{h'}(\mfu),\hat \mfu_h) \rightarrow 0,\qquad h' \uparrow h.
\end{equation}
\end{proof}

\subsection{Concatenation of trees}\label{sec.concatenation}

We summarize some properties of the concatenation of trees given in \cite{infdiv} (see also \cite{EM}).

\begin{definition}(Concatenation of trees)\label{def.concat} Let $h>0$ and 
$\mfu_{i} = [X_i,r_i,\mu_i]$, $i \in I$ ($I \subset \N \cup \{\infty\}$) be a sequence in $\UM$ with $\sum_{i \in I}\mu_i(X_i) \le C < \infty$, 
\begin{equation}\label{eq_mass1}
\overline{\mfu_i}:=\mu_i(X_i) = \sqrt{\nu^{2,\mfu_i}[0,\infty))}> 0
\end{equation} 
 and 
\begin{equation}\label{eq.def.concat}
\nu^{2,\mfu_i}(h,\infty) = 0.
\end{equation} 
We define the \it{concatenation}:
\begin{equation}
\begin{split}
\bigsqcup_{i \in I}^h \mfu_{i} := \mfu_{i_1} \sqcup_h \mfu_{i_2} \sqcup_h \ldots := \left[\biguplus_{i \in I}X_i,r^h,\sum_{i \in I}\mu_i\right],
\end{split}
\end{equation} 
where $\biguplus_{i \in I} X_i$ is the disjoint union of the $X_i$ and 
\begin{equation}\label{eq.conc.metric}
r^h(x,y) = \left\{
\begin{array}{ll}
r_i(x,y), &\quad \textrm{ for } x,y \in X_i,\\
h , &\quad \textrm{ for } x \in X_i,\ y \in X_j,\ i \not=j.
\end{array}\right.
\end{equation}
\end{definition}

\begin{definition}($h$-top)\label{def_htop} In the sense of Lemma \ref{lem_representatives} we define for $0 <h$  and $\mfu = [X,r,\mu]\in \UM$ the {\it $h$-top} $\Top_h(\mfu) \in \UM$ as
\begin{equation}
\Top_h(\mfu) := \bigsqcup\limits_{i \in \{1,\ldots,n(h)\}}^h \left[\bar B(\tau_i^h,h) ,r,\mu|_{\bar B(\tau_i^h,h) }\right] 
\end{equation}
where  $\mu|_{\bar B(\tau_i^h,h) } (\cdot) = \mu(\cdot\cap \bar B(\tau_i^h,h) )$. By Remark \ref{rem_ind} this definition is independent of the representative $(X,r,\mu)$.
\end{definition}

\begin{remark}\label{def_psi_h}
Note that $\Top_h([X,r,\mu])= [X,r^h,\mu] $ with 
\begin{equation}
r^h(x,y) = \left\{ 
\begin{array}{ll}
r(x,y),&{}\quad \textrm{if }r(x,y) \le h,\\
h,&{}\quad \textrm{otherwise}.
\end{array}\right.
\end{equation}
\end{remark}

\begin{remark}\label{Bem_fun} 
(i) Let $h> 0$ and $\mfu_i:=[\bar B(\tau_i^h,h) ,r,\mu|_{\bar B(\tau_i^h,h) }] \in \UM$, then
$\nu^{2,\mfu_i}((h,\infty) = 0$ for all $i = 1,2,\ldots,n(h)$. \par
(ii) Let $\mfu \in \UM$ and $\mfu_i$ given as in (i) for some $h > 0$. If $x \in \supp(\mu)$, then there is a $i,j \in \{1,\ldots,n(h)\}$
such that 
\begin{equation}
\mu(\bar B(x,h))  = \overline{\mfu_i} = f(h,\mfu_i) = f(h,\mfu)_{j}.
\end{equation}
(iii) As in (ii) we get for $h >0$: 
\begin{equation}
\overline{\mfu} =  \mu(X)  = \sum_{i =1}^{n(h)} f(h,\mfu)_i = \sum_{i =1}^{n(h)} \overline{\mfu_i}.
\end{equation}
Note that $\mfu \mapsto \overline{\mfu}$ is continuous, since $\overline{\mfu} = \sqrt{\nu^{2,\mfu}([0,\infty))}$. 
\end{remark}

\begin{definition} (Concatenation as partial order) 
Define for $0 <h$ the relation $\le_h$ on  $\mathbb U$ by saying $\mfu \le_h \Sk$ if there is a $\mfu' \in \UM$ with $\nu^{2,\mfu'}(h,\infty) = 0$ such that  $\Top_h(\Sk) = \Top_h(\mfu) \sqcup_h \mfu'$.
\end{definition}

\begin{lemma}\label{lem_topproperties} Let $0 < h$, $\mfu,\ \Sk \in \UM$, $(\mfu_n)$, $(\Sk_n)$ be two sequences in $\UM$ and $\UM$ be equipped with the Gromov-weak topology. 
\begin{itemize}
\item[(i)]  Suppose that $\Psi_h(\mfu_n) \rightarrow \mfu$ and 
$\chi_n := \Psi_h(\mfu_n) \sqcup_h \Top_h(\Sk_n) \rightarrow \chi$, for $n \rightarrow \infty$. Then $\mfu \le_h \chi$. 
\item[(ii)] If  $\dGP(\mfu_n, \mfu)\rightarrow 0$, 
then  $\Top_{h}(\mfu_n) \rightarrow \Top_h(\mfu)$ for all $h > 0$,  i.e. $\mfu \mapsto \Top_h(\mfu)$ is continuous. 
Moreover, if $h_n \rightarrow h$, then $\Top_{h_n}(\mfu) \rightarrow \Top_{h}(\mfu)$. 
\item[(iii)] If $\mfu_n \rightarrow \mfu$,  for $n \rightarrow \infty$, and $\Sk_n \le_h \mfu_n$, for all $n \in \N$,  then $\{\Top_h(\Sk_n):\ n\in\N\}$ is compact. 
\item[(iv)] Assume we are in the situation of Remark \ref{Bem_fun} for some $h > 0$, then 
$\nu^{2,\mfu_i} \le \nu^{2, \mfu}$ for all $i =1,\ldots,n(h)$. 
\item[(v)] Let $\mfu_n,\Sk_n \in \UM$ such that $\mfu_n \rightarrow \mfu \in \UM$ and $\Sk_n \rightarrow \Sk \in \UM$ 
and assume that $\mfu_n,\mfu,\Sk_n,\Sk$ satisfy \eqref{eq.def.concat}, then $\mfu_n\sqcup_h \Sk_n \rightarrow \mfu\sqcup_h\Sk$.
\end{itemize}
\end{lemma}

\begin{proof}
This is a summary of Proposition 2.17, Lemma 3.3 and Lemma 3.5 in \cite{infdiv}. 
\end{proof}

\section{A short note on the weak atomic topology}\label{sec.weak.atomic}

Here we give a short introduction to the weak atomic topology (see \cite{EKatomic}) and prove a Proposition that 
gives a characterization of convergence in this topology in terms of cumulative distribution functions.

\begin{definition}\label{d.weak.atomic}(Weak atomic topology) Let $(E,r)$ be a complete separable metric space and $\mu_1,\mu_2,\ldots \in \mathcal M_f(E)$ (space of finite 
Borel-measures on $E$). We say that $\mu_n \rightarrow \mu$ in the weak-atomic topology if 
\begin{itemize}
\item $\mu_n \Rightarrow \mu$ in the weak topology and 
\item $\mu_n^\ast  \Rightarrow \mu^\ast $ in the weak topology, where $\mu^{\ast} := \sum_{x \in E} \mu(\{x\})^2 \delta_x$.  
\end{itemize}
\end{definition}

\begin{proposition}\label{p.weak.atomic}
Assume that $E = \mathbb R$, and let $\mu,\mu_1,\mu_2, \ldots \in \mathcal M_f(E)$ be finite measures
with $\mu_n(E) \rightarrow \mu(E)$. Then $\mu_n \rightarrow \mu$ in the weak atomic topology 
if and only if $F_n \rightarrow F$ in the Skorohod topology on $D(\mathbb R,\mathbb R_+)$, where 
$F(t):=\mu((-\infty,t]),F_1(t):=\mu_1((-\infty,t]),F_2(t):=\mu_2((-\infty,t]),\ldots$, $t \ge 0$. 
\end{proposition}

\begin{proof}
First observe that a classical result says that $\mu_n \Rightarrow \mu$ is equivalent to $F_n(t)  \rightarrow  F(t) $ 
and $\mu_n(E) \rightarrow \mu(E)$ for all continuity points $t$ of $F$. \\

``$\Rightarrow$'' If $\mu(\{t\}) > 0$ for some $t \in \mathbb R$, then, according to Lemma 2.5 in \cite{EKatomic}, there is an unique sequence $(t_n)_{n \in \mathbb N}$ 
with $(\mu(\{t_n\}),t_n) \rightarrow (\mu(\{t\}),t)$ and all other sequences $s_n \rightarrow t$ with $s_n \neq t_n$ satisfy $\mu(\{s_n\})) \rightarrow 0$. Moreover, a simple application of the Portmanteau
Theorem gives: For all $\eps > 0$ there is a $\bar \delta > 0$ such that $\mu_n((t-\delta,t_n)\cup (t_n,t+\delta)) < \eps $ for all $n$ large enough and $\mu((t-\delta,t)\cup (t,t+\delta)) < \eps$ for all $\delta < \bar \delta$. 
If we  now choose $\delta> 0$ in such a way that $t -\delta $ is a continuity point of $F$, then
\begin{equation}
\begin{split}
\lim_{n \rightarrow \infty}&|F_n(t_n) - F(t)|  = \lim_{n \rightarrow \infty}|\mu_n((-\infty,t_n]) - \mu((-\infty,t])| \\
&\le \lim_{n \rightarrow \infty}|F_n(t-\delta) - F(t-\delta)| + \lim_{n \rightarrow \infty}|\mu_n(\{t_n\}) - \mu(\{t\})|\\
&{}\hspace{4cm}  + \limsup_{n \rightarrow \infty}|\mu_n((t-\delta,t_n)) - \mu((t-\delta,t))|\\
&\le \eps 
\end{split}
\end{equation}
and hence $F_n(t_n) \rightarrow F(t)$. A similar argument shows that the conditions of Proposition 3.6.5 in \cite{EK86} are satisfied and therefore $F_n \rightarrow F$ in the Skorohod topology. \\

``$\Leftarrow$''  Now let $F_n \rightarrow F$ in the Skorohod topology. Then for all discontinuity points $t$ of $F$ there is  one sequence $(t_n)_{n \in \mathbb N}$  such that $F(t_n) \rightarrow F(t)$ and $F(t_n-) \rightarrow F(t-)$ (see (6.20) in the proof of Proposition 3.6.5 in \cite{EK86}).  Since $\mu(\{t\}) = F(t) - F(t-)$ this gives 
\begin{equation}
\lim_{n \rightarrow \infty} \mu_n(\{t_n\}) = \lim_{n \rightarrow \infty} (F_n(t_n)-F_n(t_n-)) = F(t) - F(t-) = \mu(\{t\}) > 0. 
\end{equation}
Moreover, all other sequences $(s_n)_{n \in \mathbb N}$ with $s_n < t_n$ and  $s_n \rightarrow t$ 
satisfy $|F_n(s_n) -F(t-)| \rightarrow 0$ and hence 
\begin{equation}
\lim_{n \rightarrow \infty} \mu_n(\{s_n\}) = \lim_{n \rightarrow \infty} (F_n(s_n)-F_n(s_n-)) = 0
\end{equation}
and the analogue holds for sequences $s_n > t_n$ and $s_n \rightarrow t$. Hence we can apply Lemma 2.5 in \cite{EKatomic} and get the result. 
\end{proof}

\section{Proof of Theorem \ref{thm_polish} (a), (b)}

(a) First of all observe that 
\begin{equation}
 \dGPA(\mfu,\mfu'):=\dGP(\mfu,\mfu') + \left|\nu^{2,\mfu}(\{0\}) - \nu^{2,\mfu'}(\{0\})\right| 
 + \rho_a(\nu^{2,\mfu},\nu^{2,\mfu'})
\end{equation}
is a metric on $\UM$, where $\rho_a$ is given in \cite{EKatomic}. Now, the properties follow analogue to 
Lemma 2.3 (combined with Lemma 2.5) in \cite{EKatomic} and Proposition 5.6 in \cite{GPW09}. \\
 
(b) Recall the notation in Remark \ref{rem_Phi} and note that for $\mfu:=[X,r,\mu] \in \UM$, $\mu_h$ is purely atomic, $h > 0$. Let $A_h$ be a finite subset of $ \{x \in X:\mu_h(\{x\})> 0\}$ 
with the property, that 
\begin{equation}
 \mu_h(X\setminus A_h) < h
\end{equation}
and let $\bar \mu_h(\cdot):=\mu_h(\cdot \cap A_h)$, then, by Lemma \ref{lem_restr}  
\begin{equation}
 [X,r,\bar \mu_h] \rightarrow \mfu.
\end{equation}
In addition, note that 
\begin{equation}
 \nu^{2,\mfu}(\{h'\}) = \nu^{2,\CutTwo_h(\mfu)}(\{h'\}) \qquad \text{for all } 0 < h < h' 
\end{equation}
and 
\begin{equation}
 \left|\nu^{2,[X,r,\mu_h]}(\{h'\})-\nu^{2,[X,r,\bar \mu_h]}(\{h'\})\right| \le 2 \mu_h(X)\cdot \mu_h(X \setminus A_h).
\end{equation}
This shows $[X,r,\bar \mu_h] \rightarrow \mfu$ in the Gromov-weak atomic topology (see again Section \ref{sec.weak.atomic} or \cite{EKatomic}). \par 
Let $n \in \mathbb N$ and $a \in \mathbb R_+^n$. Using an induction argument and the fact that $\bigcup_{k \in \mathbb N}\{\sum_{i \in I} a_i^k:\ I \subset \{1,\ldots,n\}\}$ 
is countable, where $a^k \in \mathbb R_+^n$, $k \in \mathbb N$, it is straight forward 
to see that one can approximate $a$ (pointwise) by a sequences $a^k$ with 
\begin{equation}
\forall I,J \subset \{1,\ldots,n\}, \ I\cap J = \emptyset: \ \sum_{i \in I} a_i^k \neq \sum_{j \in J} a_j^k. 
\end{equation}
When we now take $[X,r,\mu] \in \UM$ with $\mu = \sum_{i = 1}^n a_i \delta_{x_i}$ for $x_i \in X$, $i = 1,\ldots,n$, this shows, 
that the sequence of measures $\mu^k = \sum_{i = 1}^n a_i^k \delta_{x_i}$ satisfy $\mu^k \Rightarrow \mu$. 
Using a similar argument as above together with Lemma \ref{l.Ielements} below, finally gives the result 
(compare also Proposition 5.6 in \cite{GPW09}). \par 
Assume $A_h = \{x_1,\ldots,x_n\}$. Then, another another way of proving this result is to disturb the measure a bit, i.e. to add a realization of independent, positive, variables $U_1,\ldots,U_n$ with $\sum_i U_i = 1/n$ to $\mu_h(\{x_1\}),\ldots,\mu(\{x_n\})$ (compare Example \ref{ex.UMI}).  \\

In order to prove (c) of this theorem, we need some more results on the function $\F$. Therefore, we skip the proof at this point and 
refer to Section \ref{sec.p.polish.c}.

\begin{remark}
Note that the above argument can be modified to prove that $\F(\UMI)$ is dense in $\F(\UM)$. 
\end{remark}

\begin{lemma}\label{l.Ielements}
Let $\mfu = [\{x_1,\ldots,x_n\},r,\mu]\in \UM$, $n \in \mathbb N \cup \{\infty\}$ and $a_i = \mu(\{x_i\})$, $i \in \mathbb N$. Then $\mfu \in \UMI$ if and only if 
\begin{equation}
\sum_{i \in I} a_i \neq \sum_{i \in J} a_i, \qquad \forall \ I,J \subset \{1,\ldots,n\},\ I \neq J.
\end{equation}
\end{lemma}

\begin{proof}
This follows directly by definition, since for all $h\ge 0$ and $x\in\{x_1,\ldots,x_n\}$ there is a set $I$ such that 
$\mu(\bar B_h(x)) = \sum_{i \in I} \mu(\{x_i\})$.
\end{proof}

\section{Proofs for Section \ref{sec.decomp}}
Here we give the proofs of our results on the function $\F$. 

\subsection{Proof of Lemma \ref{lem_representatives}}\label{sec.P.Lemma.representatives}

For the first part we observe that since $(\supp (\mu),r)$ is separable there is a countable  set $J \subset \supp (\mu)$, such that
\begin{equation}\label{lem_representatives_eq1}
\supp(\mu) \subset \bigcup_{x \in J} \bar B(x,h).
\end{equation}

We define the set 
\begin{equation}
\mathcal I:= \big\{I \subset J:\ \mu\left(\bar B(x,h) \cap \bar B(y,h)\right) = 0,\ \forall x,y \in I,\ x \neq y\big\}.
\end{equation}

Note that $\subset$ defines a partial order on $\mathcal I$. If we take a totally ordered subset $\mathcal T\subset \mathcal I$, then 
$\bigcup_{A \in \mathcal T} A \in \mathcal I$ (for two different elements $x,y \in \bigcup_{A \in \mathcal T} A$, 
there is a set $A' \in \mathcal T$, since $\mathcal T$ is totally ordered, such that $x,y \in A'$) is an upper bound for 
$\mathcal T$. By Zorn's lemma, we can find a maximal set $I \in \mathcal I$. \\

It remains to proof that 
\begin{equation}
 \mu(X) = \mu(\supp(\mu))= \sum_{x \in I} \mu(\bar B(x,h)). 
\end{equation}

Note that for $x,y \in \supp(\mu)$, since $r$ is  an ultra-metric $\mu$ almost surely, we either have
\begin{equation}
\mu\big(\bar B(x,h) \cap \bar B(y,h)\big) = 0,
\end{equation}

or 
\begin{equation}
 \mu\big(\bar B(x,h) \cap \bar B(y,h)\big) = \tilde \mu\big(\bar B(y,h)\big).
\end{equation}

By (\ref{lem_representatives_eq1}), we have 
\begin{equation}
 \mu(X) = \mu\left(\bigcup_{x \in J} \bar B(x,h)\right).
\end{equation}

If we would assume that $ \mu(X) > \sum_{x \in I}  \mu(\bar B(x,h))$, then, since $I \subset J$, we would find a 
$\tilde x \in J$ such that 
\begin{equation}
 \mu\big(\bar B(\tilde x,h) \cap \bar B(x,h)\big) = 0,\qquad \forall x \in I. 
\end{equation}

This is a contradiction, since $I$ is a maximal element of $\mathcal I$. \\

\noindent For the second part we set 
\begin{equation}\label{eq_I_def}
I_i:= \left\{j \in \{1,\ldots,\num(\delta)\}:\ \mu(\bar B(\tau_j^\delta,\delta) \cap \bar B(\tau_i^h,h))>0\right\}.
\end{equation}
Since $r$ is  an ultra-metric $\mu$-almost surely, we get $\mu(\bar B(\tau_j^\delta,\delta)\cap \bar B(\tau_i^h,h)) = \mu(\bar B(\tau_j^\delta,\delta))$ 
for all $j \in I_i$. This together with the first part implies $\ge$.\par

Let $A:= \bar B(\tau_i^h,h) \backslash \bigcup_{j \in I_i}\bar  B(\tau_j^\delta,\delta) $. If we assume that $\mu(A) > 0$, then we can take  
a $x \in A \cap \supp(\mu)$  and, by Remark \ref{rem_ind}, we find a $j$ such that
$\mu(\bar B(x,\delta) \cap\bar  B(\tau_j^\delta,\delta)) = \mu(\bar B(x,\delta))$. It follows that 
$\mu(\bar B(\tau_i^h,h) \cap \bar B(\tau_j^\delta,\delta)) = \mu(\bar B(\tau_j^\delta,\delta)) > 0$ and hence $j \in I_i$. 
A contradiction and therefore $\mu(A) = 0$. To see that $\{I_i\}_{i = 1,\ldots,\num(h)}$ forms a partition follows by similar arguments.

\subsection{A result on relative compactness and proof of Lemma \ref{l.F.cadlag} and Proposition \ref{p.tightness.GW}} \label{sec.relative.compactness.F}

We have the following result on relative compactness:

\begin{proposition}(Relative compactness and further properties)\label{thm_basic_tools}
Let $(\mfu_n)_{n \in \N}$ be a sequence in $\UM$ and $\mfu \in \UM$.
\begin{itemize}
\item[(i)] If $\mfu_n \rightarrow \mfu$ in the Gromov-weak topology and $(\h_n)_{n \in \mathbb N}$ is a sequence in $(0,\infty)$
with $\h_n \rightarrow \h \in (0,\infty)$, then $\{\f(\mfu_n,\h_n):\ n \in \mathbb N\}$ is relatively compact in 
$\mathcal S^\downarrow$, equipped with $\dSeq$.
\item[(ii)] $\F(\UM) \subset D((0,\infty),\mathcal S^\downarrow)$ and 
\begin{equation}
\lim_{h \downarrow 0} \max_i |\f(\mfu,h)_i-\f(\mfu,0)_i| = \lim_{h \downarrow 0} \dMax(\f(\mfu,h),\f(\mfu,0))= 0,
\end{equation} 
where $\f([X,r,\mu],0) := (\mu(\{x\})_{x \in X})^{\downarrow}$, i.e. the decreasing rearrangement of the atoms of $\mu$. 
\item[(iii)] If $\dSeq(\f(\mfu_n,h),\f(\mfu,h)) \rightarrow 0$ for all continuity points $h$ of $\f(\mfu,\cdot)$, then 
  $\{\mfu_n:\ n \in \N\}$ is relatively compact with respect to the Gromov-weak topology. 
\end{itemize}
\end{proposition}

\begin{remark}
Note that (ii) is Lemma \ref{l.F.cadlag}. 
\end{remark}

Before we start we need the following result on monotonicity, which is a direct consequence of Lemma \ref{lem_representatives}:

\begin{lemma}\label{lem_monoton}

Let  $0< \delta \le h$, $\mfu \in \UM$ and assume we are in the situation of Lemma \ref{lem_representatives}. Then $\num(h) \le \num(\delta)$. Moreover, for $M \le \num(h)$: 
\begin{equation}
\sum_{i=1}^M \f(\mfu,h)_i \ge \sum_{i=1}^M \f(\mfu,\delta)_i.
\end{equation}
\end{lemma}

We are now able to prove Proposition \ref{thm_basic_tools} (i).

\begin{proof} {\bf Proposition \ref{thm_basic_tools} - (i)} 
Note that if we equip $\mathcal S^\downarrow_C$, $C > 0$ with 
\begin{equation}\label{eq.dMax.def}
\dMax(x,y):= \max\{|x_i-y_i|:\ i \in \N\},\qquad x,y \in \mathcal S^\downarrow_{C},
\end{equation}

then $(\mathcal S^\downarrow_C,\dMax)$ is a compact space (this follows analogue to Proposition 2.1. in \cite{B}). \par 
Note that $\mfu_n \rightarrow \mfu$ implies $\overline{\mfu_n} \rightarrow \overline{\mfu}$, where $\overline{\mfu} = \sqrt{\nu^{2,\mfu}([0,\infty))}$ and hence we find a constant $C > 0$ 
such that 
\begin{equation}
\sup_{n \in \N} \overline{\mfu_n} = \sup_{n \in \N} \sum_{i \in \N} \f(\mfu_n,h_n)_i \le C.
\end{equation}

It follows that $\{\f(\mfu_n,h_n):\ n \in \N\}$ is relatively compact in $(\mathcal S^\downarrow_{C},\dMax)$. Hence, there is a $x \in \mathcal S^\downarrow_{C}$ such that $\dMax(\f(\mfu_{n_k},h_{n_k}),x) \rightarrow 0$ along some subsequence  and we have to show that 
$\dSeq(\f(\mfu_{n_k},h_{n_k}),x) \rightarrow 0$. We suppress the dependence on the subsequence and set
\begin{equation}
\delta := \inf_{n \in \N} h_n > 0.
\end{equation}

Next we prove that for all $0 < \e \le \delta$ there is a $M \in \mathbb N$ such that  
\begin{equation}\label{eq_uni}
\sup_{n \in \mathbb N}\sum_{i = M+1}^\infty \f(\mfu_n,h_n)_i < \e.
\end{equation}

We assume the converse, i.e. assume there is an $\e >0 $ with $\e \le \delta$ such that for all $M \in \mathbb N$ there is a $n \in \mathbb N$ with 
\begin{equation}\label{eq_AAA1}
\sum_{i=M+1}^{\infty} \f(\mfu_n,h_n)_i \ge \e.
\end{equation}
Note that $\f(\mfu_n,\bar \e)_i \le \frac{C}{M}$ for all 
$i \ge M$, $M \in \N$, $\bar \e > 0$. Moreover, note that when $\bar \e \le \delta$, Lemma \ref{lem_monoton}  implies
\begin{equation}
\sum_{i = M+1}^\infty \f(\mfu_n,\bar \e)_i \ge \e.
\end{equation}
Since $[X_n,r_n,\mu_n] = \mfu_n \rightarrow \mfu$, we have (see Proposition 7.1 
in \cite{GPW09} and Proposition B.2 in \cite{infdiv}):
\begin{equation}
\begin{split}
0 &= \lim_{M \rightarrow \infty} \sup_{n \in \N} \nu_{\frac{C}{M}}(\mfu_n) \\
&= \lim_{M \rightarrow \infty} \sup_{n \in \N} \ \inf \left\{\bar \e > 0:\ \mu_n\left(\left\{x\in X_n: \mu_n(B^{r_n}(x,\bar \e)) 
\le \frac{C}{M}\right\} \right)\le \bar \e\right\}\\
&\ge \lim_{M \rightarrow \infty} \sup_{n \in \N} \ \inf \left\{\bar \e > 0:\ \mu_n\left(\left\{x\in X_n: \mu_n(\bar B^{r_n}(x,\bar \e)) 
\le \frac{C}{M}\right\} \right)\le \bar \e\right\}\\
&=  \lim_{M \rightarrow \infty} \sup_{n \in \N} \ \inf \left\{\bar \e > 0:\ \sum_{i = 1}^\infty \f(\mfu_n,\bar \e)_i 1\left(\f(\mfu_n,\bar \e)_i \le \frac{C}{M}\right)  \le \bar \e \right\} \\
&\ge  \lim_{M \rightarrow \infty} \sup_{n \in \N} \ \inf \left\{\bar \e > 0:\ \sum_{i = M+1}^\infty \f(\mfu_n,\bar \e)_i  
\le \bar \e \right\} \\
&\ge \e,
\end{split}
\end{equation}
a contradiction and (\ref{eq_uni}) follows. \\

If we now define for $\e >0$
\begin{equation}
M^x:= \min\Big\{K:\ \sum_{i=K+1}^\infty x_i < \e \Big\},
\end{equation}
then (\ref{eq_uni})  implies for all $0<\e \le \delta$, there is a $M \in \N$ such that 

\begin{equation}
\sum_{i=1}^{\infty}\big|\f(\mfu_n,h_n)_i - x_i \big| \le (M \vee M^x ) \cdot \max_{i \in \N} \big|\f(\mfu_n,h_n)_i - x_i\big| + 
2\e,\qquad 	\forall n \in \mathbb N.
\end{equation}

Therefore
\begin{equation}
\sum_{i=1}^{\infty}\big|\f(\mfu_n,\h_n)_i - x_i \big|\stackrel{n \rightarrow \infty}{\longrightarrow} 0.
\end{equation}
\end{proof}

\begin{remark}\label{rem_equinorms}
The above proof also shows: If $x\in \mathcal S^\downarrow_C$,  $\mfu_n \rightarrow \mfu$ in the Gromov-weak topology and $0<h_n \rightarrow h > 0$, then the following is equivalent: 
\begin{itemize}
\item[(i)] $\dSeq\big(\f(\mfu_n,h_n),x\big) \rightarrow 0$,
\item[(ii)] $\dMax\big(\f(\mfu_n,h_n),x\big) \rightarrow 0$.
\end{itemize}
Moreover, by Proposition 2.1. in \cite{B}, the above is equivalent to
\begin{itemize}
\item[(iii)] $\f(\mfu_n,h_n)_i \rightarrow x_i $ for all $i \in \mathbb N$.
\end{itemize}
\end{remark}

Next we show that $\F$ takes values in the space of cadlag functions. Before we do that we need the following Lemma:

\begin{lemma}\label{l.Gwa.embedding}
Let $\mfu = [X,r,\mu], \mfu_1 = [X_1,r_1,\mu_1],\mfu_2 = [X_2,r_2,\mu_2],\ldots \in \UM$ and assume that $\mfu_n \rightarrow \mfu$ in the Gromov-weak topology 
and $\nu^{2,\mfu_n}(\{0\}) \rightarrow \nu^{2,\mfu}(\{0\}) $. Then there is a complete separable metric space $(Z,r_Z)$ and isometric embeddings $\varphi_n:X_n\to Z$, $n \in \mathbb N$ and $\varphi:X \to Z$ such that 
$\mu_n \circ \varphi^{-1}_n \rightarrow \mu \circ \varphi^{-1}$ in the weak atomic topology.  
\end{lemma}

\begin{proof}
Observe that 
\begin{equation}
\mu^\ast_n := \sum_{x \in X_n} \mu_n(\{x\})^2 = \nu^{2,\mfu_n}(\{0\}) \rightarrow \nu^{2,\mfu}(\{0\}) = \sum_{x \in X} \mu(\{x\})^2.
\end{equation}
When we now apply Lemma 5.8 in \cite{GPW09} combined with Lemma 2.1 and 2.2 in \cite{EKatomic}, the result follows. 
\end{proof}

\begin{proof} {\bf Proposition \ref{thm_basic_tools} - (ii)} 
Let $h \in (0,\infty)$ and $(h_n)_{n \in \N}$ be a sequence in $(0,\infty)$ with $h_n \downarrow h$ and $\mfu = [X,r,\mu]$. Let $\{\rep_i^h: i \in \{1,\ldots,n(h)\}$ be as in Lemma \ref{lem_representatives}. Then we have 
\begin{equation}
\bar B(\rep_i^h,h) = \bigcap_{n \in \mathbb N} \bar B(\rep_i^h,h_n),\qquad \forall i = 1,\ldots,n(h) 
\end{equation}
and the $\sigma$-continuity of the measure $\mu$ gives $\f(\mfu,h_n)_i \rightarrow \f(\mfu,h)_i$ for all $i \in \{1,\ldots,n(h)\}$. 
By Remark \ref{rem_equinorms}, this is enough to get the right continuity. \\

In order to prove that the limits from the left exist, assume $h_n \uparrow h$. Note that by Proposition \ref{thm_basic_tools} (i), 
$\{\f(\mfu,h_n):\ n \in \N\}$ is relatively compact in $(\mathcal S^\downarrow,\dSeq)$. Let $\tilde x \in \mathcal S^\downarrow$ be a limit point along some subsequence $(n_k)_{k \in \N}$, $\e >0$ and $M \in \{1,\ldots,n(h)\}$. By Lemma \ref{lem_monoton} we have for $m$ large enough
\begin{equation}
\sum_{i = 1}^M \f(\mfu,h)_i\ge \sum_{i = 1}^M \f(\mfu,h_n)_i \ge \sum_{i = 1}^M \f(\mfu,h_m)_i,\qquad \forall n \ge m.
\end{equation}
Hence $0 \le \sum_{i = 1}^M \f(\mfu,h_n)_i$ is a monotonically increasing sequence and therefore it converges to some $S_M \ge 0$. 
It follows that $\f(\mfu,h_n)_1 \rightarrow S_1$ and for $M > 1$
\begin{equation}
\f(\mfu,h_n)_M = \sum_{i = 1}^{M} \f(\mfu,h_n)_i - \sum_{i = 1}^{M-1} \f(\mfu,h_n)_i  \stackrel{n \rightarrow \infty}{\longrightarrow} S_{M} - S_{M-1}, 
\end{equation}
i.e. $\tilde x_i = S_{i} - S_{i-1}$ ($S_0 := 0$) for all $i \in \{1,\ldots,n(h)\}$ independent of the subsequence and the existence of the left limit follows. \\ 

For the second part, set  $\mfu_n := \Cut_{h_n}(\mfu) = [X_n,r_n,\mu_n]$, for $h_n \downarrow 0$. Then, by Lemma \ref{lem_restr}, $\mfu_n \rightarrow \mfu$, in the Gromov weak topology. Moreover, by the definition of $\mu_n$ (see Remark \ref{rem_Phi}), we have
\begin{equation}
\nu^{2,\mfu_n}(\{0\}) = \nu^{2,\mfu}([0,t_n]) \rightarrow \nu^{2,\mfu}(\{0\}). 
\end{equation} 
We can now apply Lemma 2.5 in \cite{EKatomic} combined with Lemma \ref{l.Gwa.embedding} to get the result.
\end{proof}

\begin{remark}\label{r.weak.atomic.approx}
 Let $\mu_h$ be the measure given in Remark  \ref{rem_Phi}. As we have seen in the above proof, $\mu_h \rightarrow \mu$ in the weak atomic topology when $h \downarrow 0$ (on $(X,r)$). In particular, by Definition \ref{d.weak.atomic}, 
 \begin{equation}
  \sum_{i \in \mathbb N} f(\mfu,h)_i^2 \rightarrow  \sum_{i \in \mathbb N} f(\mfu,0)_i^2,\qquad \text{for } h \downarrow 0. 
 \end{equation}
 Note also that 
 \begin{equation}
  \nu^{2,\mfu}(\{0\}) = \sum_{i = 1}^\infty \f(\mfu,0)_i^2.
  \end{equation}

\end{remark}

Before we prove Proposition \ref{thm_basic_tools}  (iii), we need the following connection of $\F(\mfu)$ for a given $\mfu \in \UM$ and the pairwise distance matrix distribution $\nu^{2,\mfu}$: 

\begin{lemma}\label{l.connect.f.nu}
	Let $\mfu = [X,r,\mu] \in \UM$, $\h > 0$.  Then
	\begin{equation}
		\nu^{2,\mfu}[0,\h] = \sum_{i \in \mathbb N} (\f(\mfu, \h)_i)^2 
	\end{equation}
	and $\nu^{2,\mfu}\{\h\} = 0$ if and only if  $ \dSeq(\f(\mfu,\h-),\f(\mfu,\h)) = 0$, i.e. the continuity points of $\nu^{2,\mfu}$ 
	are exactly the ones of $\f(\mfu,\cdot)$. 
\end{lemma}

\begin{remark}\label{rem.nu.atomic}
 Since $\f(\mfu,\cdot)$ is cadlag and constant between its jump points, the measure $\nu^{2,\mfu}$ is purely atomic.
\end{remark}

\begin{proof}
	Recall Lemma \ref{lem_representatives}. By definition we have
	\begin{equation}\label{eq.connect.0}
	\begin{split}
		\nu^{2,\mfu}[0,h] &= \mu\otimes \mu \big(\big\{(x,y)\in X\times X: r(x,y)  \le h\big\}\big) \\
		&= \sum_{i \in \{1,\ldots,\num(h)\}} \mu(\bar B(\rep_i^h,h))^2 \\
		&=  \sum_{i \in \N} (\f(\mfu, h)_i)^2.
	\end{split}
	\end{equation}
	Moreover, since $\F(\mfu)$ is cadlag, where $\mathcal S^\downarrow$ is equipped with $\dSeq$, we get
	\begin{equation}\label{eq.connect.1}
		\nu^{2, \mfu}[0,h)  = \lim_{\e \downarrow 0}\nu^{2, \mfu}[0,h-\e] = \sum_{i \in \N} (\f(\mfu, h-)_i)^2.
	\end{equation}
	It follows that
	\begin{equation}\label{eq.connect.2}
		\nu^{2,\mfu}\{h\} = \sum_{i \in \N} \left((\f(\mfu, h)_i)^2 - (\f(\mfu, h-)_i)^2\right).
	\end{equation}
	As a consequence we get (recall $\overline{\mfu}:=\sqrt{\nu^{2,\mfu}([0,\infty))}$ is the total mass.)
	\begin{equation}
	\begin{split} \label{rem_aa1}
		\nu^{2,\mfu}\{h\}&= \sum_{i \in \N}  \left((\f(\mfu, h)_i)^2 - (\f(\mfu, h-)_i)^2\right) \\
		&\le \sum_{i \in \N} \big|\f(\mfu, h)_i - \f(\mfu, h-)_i\big|\big(\f(\mfu, h)_i + \f(\mfu, h-)_i\big) \\
		&\le 2\overline{\mfu}\cdot  \dSeq\big(\f(\mfu,h),\f(\mfu,h-)\big). 
	\end{split}
	\end{equation}
	and ``$\Leftarrow $'' follows. Now let $\{\repp_i^h: \ i \in \{1,2,\ldots,\num'(h)\}\}$ be as in Lemma \ref{lem_representatives}  (for $(X,r,\mu)$)  where we
	replaced the closed balls by open balls (see Remark \ref{rem_pos}). Following the proof of Lemma \ref{lem_representatives}
	we find for each $i \in \{1,\ldots,\num(h)\}$ a set 
	$I_{i} \subset \{1,2,\ldots,\num'(h)\}$ such that
	\begin{equation}
		\mu(\bar B(\rep_{i}^h,h)) = \sum_{j \in I_{i}} \mu(B(\repp_{j}^h,h))
	\end{equation}
	and $\{I_i\}_{i = 1,\ldots,\num(h)}$ is a partition of $\{1,2,\ldots,\num'(h)\}$. 	Now, using (\ref{eq.connect.0}) and (\ref{eq.connect.1}), we get 
	\begin{equation}
		\sum_{i \in \N} \f(\mfu, \h-)_i^2 = \sum_{i = 1}^{\num'(\h)}   \mu(B(\repp_{i}^\h,\h))^2
	\end{equation}
	and $\nu^{2,\mfu}(\{\h\}) = 0$ implies that	
	\begin{equation}
	\sum_{i = 1}^{\num'(\h)}   \mu(B(\repp_{i}^\h,\h))^2 = \sum_{i = 1}^{\num(\h)}\mu(\bar B(\rep_{i}^h,h))^2 = 
	\sum_{i = 1}^{\num(\h)} \left(\sum_{j \in I_i}\mu( B(\repp_{j}^h,h))\right)^2. 
	\end{equation}	
	But this is equivalent to 	
	\begin{equation}
	\sum_{i = 1}^{\num(\h)}   \left(\sum_{j \in I_i}\mu(B(\repp_{j}^\h,\h))^2 -  \left(\sum_{j \in I_i}\mu( B(\repp_{j}^h,h))\right)^2\right) = 0. 
	\end{equation}
	Since 	
	\begin{equation}
	\sum_{j \in I_i}\mu(B(\repp_{i}^\h,\h))^2 -   \left(\sum_{j \in I_i}\mu( B(\repp_{j}^h,h))\right)^2 \le 0,
	\end{equation}
	we get
	\begin{equation}
	\sum_{j \in I_i}\mu(B(\repp_{i}^\h,\h))^2 = \left(\sum_{j \in I_i}\mu( B(\repp_{j}^h,h))\right)^2,
	\end{equation}
	for all $i \in \{1,\ldots,\num(h)\}$ and hence $|I_i| = 1$ (note that $\mu( B(\repp_{j}^h,h)) > 0$ for all  $j$ - see Remark \ref{rem_pos}). But this shows that $\mu(\bar B(\rep_{i}^\h,\h)) =  \mu(B(\repp_{i}^\h,\h))$ for all $i$ and by definition of $\f$ (as the reordering of such masses) the result follows.
\end{proof}

We can now finish the proof of Proposition \ref{thm_basic_tools}:

\begin{proof} {\bf Proposition \ref{thm_basic_tools} - (iii)} 	To prove relative compactness recall that  
	\begin{align}
		\nu^{2,\mfu}[0,\h] &= \sum_{i \in \mathbb N} (\f(\mfu, \h)_i)^2, \\
		\nu^{2,\mfu}([0,\infty)) & = \sqrt{\sum_{i \in \mathbb N} \f(\mfu, \h)_i}, 
	\end{align}
	for all $\h > 0$ and that $\h > 0$ is a continuity point of $\f(\mfu,\cdot)$  iff $\nu^{2,\mfu}[0,\h):= \nu^{2,\mfu}[0,\h]$, by Lemma \ref{l.connect.f.nu}.
	
	Since $\dSeq(\f(\mfu_n,\h),\f(\mfu,\h)) \rightarrow 0$ for all continuity points $\h$ of $\f(\mfu,\cdot)$ and 
\begin{equation}
	0\le \limsup_{n \rightarrow \infty} \nu^{2,\mfu_n} (\{0\}) \le \nu^{2,\mfu}([0,\delta])
	\end{equation}
	for all continuity points $\delta>0$, we get
	\begin{equation}\label{eq.conv.pd}
	\lim_{n \rightarrow \infty} \nu^{2,\mfu_n}(\{0\})  = \limsup_{n \rightarrow \infty} \nu^{2,\mfu_n}(\{0\}) = 0, 
	\end{equation}
	if $\nu^{2,\mfu}\{0\} = 0$. This gives
	\begin{equation}
		\nu^{2,\mfu_n} \stackrel{n \rightarrow \infty}{\Longrightarrow} \nu^{2,\mfu}.
	\end{equation}

	For the relative compactness of $\{\mfu_n:\ n \in \N\}$ it remains to show, that (see Theorem 2 and Remark 2.11 in \cite{GPW09}; see also Proposition B.2 in \cite{infdiv}):
	\begin{equation}\label{bew_cor_gl1}
		\lim_{\delta \rightarrow 0} \ \limsup_{n \rightarrow \infty} \nu_{\delta}(\mfu_n) = 0,
	\end{equation}	
	where $\nu_{\delta}(\cdot)$ is the modulus of mass distribution: 
	\begin{equation}\label{eq.modOfMass}
		\nu_{\delta}([X,r,\mu]) = \inf \Big\{\bar \e > 0:\ \mu\big(\{x\in X: \mu(B(x,\bar \e)) \le \delta\}\big) \le \bar \e\Big\}.
	\end{equation}

	Note that if $\mfu_n = [X_n,r_n,\mu_n]$, then (see Lemma \ref{l.connect.f.nu} and its proof)
	\begin{equation}
	\begin{split}
		\mu_n\big(\{x\in X_n: &\mu_n(B^{r_n}(x,2\e)) \le \delta\}\big) \\
		&\le \mu_n\big(\{x\in X_n: \mu_n(\bar B^{r_n}(x,\e)) \le \delta\}\big)\\
			&=	\sum_{i =1}^{\infty} \f(\mfu_n, \e)_i \cdot \mathds{1}(\f(\mfu_n, \e)_i \le \delta),
	\end{split}
	\end{equation}
	for all $\e > 0$. Define 
	\begin{equation}
	M(\delta;\e):= \min\{i \in \mathbb N:\ f(\mfu,\e)_i < \delta\}.
	\end{equation}
	If $\e$ is a continuity point of $\f(\mfu, \cdot)$ we get $\dSeq(\f(\mfu_n, \e),\f(\mfu,\e)) \rightarrow 0$ and hence 
	\begin{equation}
	\begin{split}
	\limsup_{n \rightarrow \infty}  \sum_{i =1}^{\infty} \f(\mfu_n,\e)_i \cdot \mathds{1}(\f(\mfu_n, \e)_i < \delta) 
	&= \limsup_{n \rightarrow \infty} \sum_{i = M(\delta; \e)}^\infty \f(\mfu_n,\e)_i \\
	&= \sum_{i = M(\delta; \e)}^\infty \f(\mfu,\e)_i. 
	\end{split}
	\end{equation}
	Since the right hand side converges to $0$ when $\delta \downarrow 0$, the result follows.
\end{proof}

Finally we give the proof of Proposition \ref{p.tightness.GW}: 

\begin{proof}(Proposition \ref{p.tightness.GW})
First observe that, by assumption, 
\begin{equation}
\begin{split}
\limsup_{n \rightarrow \infty} &P\left(\nu^{2,\U^n}([H,\infty)) \ge \eps\right) \\
&= \limsup_{n \rightarrow \infty} P\left(\left(\sum_{i = 1}^\infty \f(\U^n,H)_i\right)^2-\sum_{i = 1}^\infty \f(\U^n,H)^2_i \ge \eps\right) \\
&\le \eps.
\end{split}
\end{equation}
Hence, $\left(\mathcal L(\nu^{2,\U^n})\right)_{n \in \mathbb N}$ is tight. 
On the other hand, recall the definition of $\nu_\delta$ 
(see \eqref{eq.modOfMass}) and note that by condition (i), for all 
 $\delta > 0$ and $\eps > 0$ there is a $M$ such that 
\begin{equation}
\limsup_{n \rightarrow \infty} P(\sum_{i = M}^\infty \f(\U^n,\delta) \ge \eps) \le \eps.
\end{equation}
Therefore (compare the proof of Proposition \ref{thm_basic_tools} - (iii)), for all $\eps > 0$ there is a $\delta > 0$ such that 
\begin{equation}
\begin{split}
 \limsup_{n \rightarrow \infty} &P\left(\nu_\delta(\U^n) \ge \eps\right) \\
&\le \limsup_{n \rightarrow \infty}  P\left(\sum_{i =1}^{\infty} \f(\mfu_n, \e)_i \cdot \mathds{1}(\f(\mfu_n, \e)_i \le \delta) \ge \eps \right) \\
&=  \limsup_{n \rightarrow \infty} P\left(\sum_{i = M}^\infty \f(\mfu,\e)_i \ge \eps \right)\\
&\le \eps.
\end{split}
\end{equation}
Hence, the result follows by Theorem 3 in \cite{GPW09} (compare also Remark 3.2 and Remark 7.2 (ii) in \cite{GPW09}). 
\end{proof}

\subsection{Proof of Theorem \ref{thm.perfect} and Theorem \ref{thm.subspace.homeomorphism}}

The main ingredient for the proof of continuity in Theorem \ref{thm.perfect} is the following lemma.

\begin{lemma}\label{l.conv.contpoint}
Let $(\mfu_n)_{n \in \N}$ be a sequence in $\UM$ and $\mfu \in \UM$.
If $\dGP(\mfu_n,\mfu) \rightarrow 0$ then $\dSeq(\f(\mfu_n,h),\f(\mfu,h)) \rightarrow 0$ for all continuity points $h$ of $\f(\mfu,\cdot)$.
\end{lemma}
Before we can prove this Lemma we need the following: 
\begin{lemma} \label{lem_uni}
Recall the notation of Section \ref{sec.concatenation} and let $\UM$ be equipped with the Gromov-weak topology. 
Let $(\mfu_n)_{n \in \N}$ be a sequence in $\UM$, $\mfu \in \UM$ and $\mfu_i^n $, $\mfu_i$ be  as in Remark \ref{Bem_fun}. Moreover, let 
$\hat \mfu_n := \bigsqcup_{i \in J_n}^h \mfu_i^n$ and assume that $|J_n| \equiv C \in \N$. If $h>0$ is a continuity point of $f(\mfu,\cdot)$, 
$\hat \mfu_n \rightarrow \hat \mfu$ and $\mfu_n \rightarrow \mfu$ in the Gromov-weak topology, then  $\hat \mfu \le_h \Top_h(\mfu)$ and $\hat \mfu  = \bigsqcup_{i \in J}^h \mfu_i \in \UM $ with $|J| \le C$. 
\end{lemma}

\begin{proof}
Note that $\hat \mfu_n \le_h \Top_h(\mfu_n)$ and hence Lemma \ref{lem_topproperties} (i) and (ii) implies $\hat \mfu \le_h \Top_h(\mfu)$.
Moreover, by Lemma \ref{l.connect.f.nu} and Lemma \ref{lem_topproperties} (iv) we get
\begin{equation}
\nu^{2,\mfu_i}([h,\infty)) = 0,\qquad \forall i = 1,\ldots,n(h)
\end{equation}
and hence Theorem 2.13 in \cite{infdiv} implies the existence of a set $J$ such that
\begin{equation}
\hat \mfu = \bigsqcup^h_{i \in J} \mfu_i.
\end{equation}

If we take a sequence $i_n \in J_n$, then  Lemma \ref{lem_topproperties} (iii) and (i) gives the existence of a subsequence such that $\mfu^{n_k}_{i_{n_k}} \rightarrow \tilde \mfu \le_h \hat \mfu$ and hence, there is a subset $\tilde J \subset J$ such that 
\begin{equation}
\tilde \mfu = \bigsqcup^h_{i \in \tilde J} \mfu_i.
\end{equation}

We assume that 
\begin{equation}
|\tilde J| \ge 2.
\end{equation}
Then there are $j_1,j_2 \in \tilde J$ such that (see Lemma \ref{l.connect.f.nu}, Remark \ref{Bem_fun} and Lemma \ref{lem_representatives}):
\begin{equation}
\nu^{2,\tilde \mfu}(\{h\}) \ge \overline{\mfu_{j_1}}\cdot \overline{\mfu_{j_2}} >0
\end{equation}
and for every $\delta > 0$, by the Portmanteau Theorem:
\begin{equation}
\begin{split}
0 &< \nu^{2,\tilde \mfu}(\{h\}) \le \nu^{2,\tilde \mfu}(h-\delta,h+\delta) \le \liminf_{k \rightarrow \infty}\nu^{2,\mfu^{n_k}_{i_{n_k}}}(h-\delta,h+\delta)\\
& \stackrel{Lem \ \ref{lem_topproperties} (iv)}{\le} \liminf_{k \rightarrow \infty}\nu^{2, \mfu_{n_k}}(h-\delta,h+\delta)\\
&\le \limsup_{k \rightarrow \infty}\nu^{2, \mfu_{n_k}}[h-\delta,h+\delta] \le \nu^{2, \mfu}[h-\delta,h+\delta].
\end{split}
\end{equation}
Since $\bigcap_{\delta >0}[h-\delta,h+\delta]= \{h\}$, this is a contradiction to Lemma \ref{l.connect.f.nu}. 
Hence $|\tilde  J| \le 1$ and therefore there is either one $i \in J$ such that
\begin{equation}
\mfu_{n_k}^{i_{n_k}}\rightarrow \mfu_i,\\[-0.3cm]
\end{equation}
or 
\begin{equation}\label{eq.conv.neutral}
\mfu_{n_k}^{i_{n_k}} \rightarrow [\{0\},0,0]=:e,
\end{equation}
the neutral element with respect to $\bigsqcup$, i.e. $\mfu \sqcup e = \mfu = e \sqcup \mfu$. Combining this with Lemma \ref{lem_topproperties} (iii) and (v) implies $|J| \le C$.  
\end{proof}

\begin{remark}\label{r.num.lower}
In fact, the above argument shows that the number of balls map, i.e. $\num(h): \UM \to\mathbb N \cup \{\infty\}$, for $h > 0$ (see Lemma \ref{lem_representatives}), is lower semi-continuous, when $h$ is a continuity point.
\end{remark}

Now we can prove Lemma \ref{l.conv.contpoint}: 
\begin{proof} (Lemma \ref{l.conv.contpoint})
Let $(n_k)_{k \in \N}$ such that $\dSeq(\f(\mfu_{n_k},h), x) \rightarrow 0 \in \mathcal S^\downarrow$ 
(see Proposition \ref{thm_basic_tools} (i)). We will now prove that $x = \f(\mfu,h)$ and therefore we may assume w.l.o.g. that $\f(\mfu_{n},h) \rightarrow x$ for the following. Define
\begin{equation}
\begin{split}
&\tilde x_1 := x_1,\\
&\tilde x_l := \max(\{x_i: i \in \N\}\backslash\{\tilde x_1,\ldots,\tilde x_{l-1}\}),\ l \ge 1.
\end{split}
\end{equation}
Note that $\dSeq(\f(\mfu_{n},h), x) \rightarrow 0$ implies: For all $L \in \N$ there is a $\bar \e > 0$ and $N\in \mathbb N$ 
such that
\begin{equation}
\{i \in \mathbb N:\ |\f(\mfu_n,h)_i - \tilde x_l| < \e\}  = \{i \in \N:\ x_i = \tilde x_l\} =: C_l,
\end{equation}
for all $n \ge N$, $0 < \eps < \bar \eps$ and $l \le L$. Let $(\mfu_i^n)_{i \in \mathbb N}$ and $(\mfu_i)_{i \in \mathbb N}$ be as in Remark \ref{Bem_fun}, $\delta > 0$ and $L \in \mathbb N$ large enough such that 
\begin{equation}
\limsup_{n \rightarrow \infty} \left|\sum_{l \le L} |C_l| \tilde x_l -  \sum_{i \in \mathbb N} \f(\mfu_n,h)_i \right| \le \delta
\end{equation}
 and define 
\begin{equation}
U^l_n:= \{i \in \N:\ |\overline{\mfu_i^n} - \tilde x_l| < \e\}.
\end{equation}
Then, by Remark \ref{Bem_fun} and for all $n$ large enough,
\begin{equation}
|U^l_n| = |\{i \in \N:\ |\f(\mfu_n,h)_i - \tilde x_l| < \e\} | = |C_l|\in \mathbb N
\end{equation}

and  by Lemma \ref{lem_uni} and Lemma \ref{lem_topproperties} (iii) and (v), we can find a subsequence $(n_k)_{k \in \N}$ and a set $\hat U^l \subset \N$ with $|\hat U^l| \le |C_l|$ such that 
\begin{equation}
\{\mfu_i^{n_k}:\ i \in U_{n_k}^l\} \rightarrow \{\mfu_i:\ i \in \hat U^l\}
\end{equation}
(where we allow duplications in the above sets) and 
\begin{equation}
\bigsqcup_{i \in U_{n_k}^l} \mfu_i^{n_k} \rightarrow \bigsqcup_{i \in \hat U^l} \mfu_i, 
\end{equation}
for all $l \le L$. Let 
\begin{equation}
U^l:= \{i \in \N:\ \overline{\mfu_i} = \tilde x_l\}.
\end{equation}
Since $\overline{\mfu_{i_n}^n} \rightarrow \tilde x_l$  for every sequence $i_n \in U^l_n$ we get  $\hat U^l \subset U^l$ (independent of the choice of $(n_k)$) for all $l \le L$ and, by the observation in \eqref{eq.conv.neutral}, we have $|\hat U^l| = |C_l|$. On the other hand, if there is a $l^\ast\le L$ such that $|U^{l^\ast}| > |C_{{l^\ast}}|$ then 
\begin{equation}
\begin{split}
\overline{\mfu} &= \lim_{n \rightarrow \infty} \overline{\mfu}_n = \lim_{n \rightarrow \infty} \sum_{l \in \mathbb N} f(\mfu_n,h)_l \\
&\le  \lim_{n \rightarrow \infty} \sum_{l \le L} |C_l| \tilde x_l + \delta
= \lim_{k \rightarrow \infty} \sum_{l \le L} \sum_{i \in U^l_{n_k}} \overline{\mfu_{i}^{n_k}} + \delta\\
&= \sum_{l \le L} \sum_{i \in \hat U^l} \overline{\mfu_{i}} + \delta
\le \sum_{l \le L} \sum_{i \in U^l} \overline{\mfu_{i}} - \tilde x_{l^\ast} + \delta \\
&\le \overline{\mfu} - \tilde x_{l^\ast} + \delta, 
\end{split}
\end{equation}
which implies $|U^l| = |C_l|$ for all $l \le L$ with $\tilde x_l \ge \delta$. Hence, 
$\hat U^l = U^l$ for all those $l$ and therefore, letting $\delta \downarrow 0$, $x = f(\mfu,h)$.
\end{proof}

In order to prove Theorem \ref{thm.perfect}, we need the following Lemma. 

\begin{lemma}\label{lem_fnu}
	Let  $(x_n)_{n \in \N}$, $(y_n)_{n \in \N}$ be a sequence in $(0,\infty)$ with $x_n \rightarrow x >0$, $y_n \rightarrow y >0$ for 
	$n \rightarrow \infty$ and $x_n < y_n$ for all $n \in \N$. Moreover, let $(\mfu_n)_{n \in \N}$ be a sequence in $\UM$ with 
	$\mfu_n \rightarrow \mfu \in \UM$, for $n \rightarrow \infty$ (with respect to the Gromov-weak topology). Then 
	\begin{itemize}
	\item[(i)] $\nu^{2,\mfu_n}(x_n, y_n] \rightarrow 0$ iff $\dSeq\Big(\f\big(\mfu_n,x_n\big),\f\big(\mfu_n,y_n\big)\Big) \rightarrow 0$,
			for $n \rightarrow \infty$.
	\item[(ii)] $\nu^{2,\mfu_n}[x_n, y_n) \rightarrow 0$ iff $\dSeq\Big(\f\big(\mfu_n,x_n-\big),\f\big(\mfu_n,y_n-\big)\Big) \rightarrow 0$,
			for $n \rightarrow \infty$.
	\item[(iii)] $\nu^{2,\mfu_n}(x_n, y_n) \rightarrow 0$ iff $\dSeq\Big(\f\big(\mfu_n,x_n\big),\f\big(\mfu_n,y_n-\big)\Big) \rightarrow 0$,
			for $n \rightarrow \infty$.
	\end{itemize}
\end{lemma}

\begin{proof}
	Similar to (\ref{rem_aa1}) we get ``$\Leftarrow$'' and it remains to prove ``$\Rightarrow$''. By Lemma \ref{l.connect.f.nu} we have
	\begin{equation}
		\nu^{2,\mfu_n}(x_n, y_n] = \sum_{i \in \N} \Big((\f(\mfu_n, x_n)_i)^2 - (\f(\mfu_n, y_n)_i)^2\Big)
	\end{equation}
	and by Lemma \ref{lem_representatives} we find for all $n \in \N$ a partition $\{I_i^n\}_{i = 1,\ldots,N(y_n)}$ of $\{1,\ldots,N(x_n)\}$ (where we write $N(h)$ instead of $n(h)$ to avoid confusion with the index of the sequence $\mfu_n$)
	such that 
	\begin{equation}\label{eq.p.fnu.1}
		\begin{split}
			\sum_{i \in \N} &\Big((\f(\mfu_n, y_n)_i)^2 - (\f(\mfu_n, x_n)_i)^2\Big)\\
			&= \sum_{i = 1}^{N(y_n)} \Big(\sum_{j \in I^n_i} \f(\mfu_n, x_n)_j\Big)^2 
					- \sum_{i = 1}^{N(x_n)}(\f(\mfu_n, x_n)_i)^2\Big) \\
			&= \sum_{i = 1}^{N(y_n)}  \sum_{k,l \in I^n_i,\ k \neq l} \f(\mfu_n, x_n)_k\cdot \f(\mfu_n, x_n)_l\\
			&= 2 \sum_{i = 1}^{N(y_n)}  \mathds{1}(|I^n_i| \ge 2) \sum_{k,l \in I^n_i,\ k < l} \f(\mfu_n, x_n)_k\cdot \f(\mfu_n, x_n)_l \\
			&\stackrel{n \rightarrow \infty}{\longrightarrow } 0. 
		\end{split}
	\end{equation}	
	We apply Proposition \ref{thm_basic_tools} (i) twice and find $a,b \in \mathcal S^\downarrow$ such that 
	\begin{equation}
		\dSeq\Big(\f\big(\mfu_{n_k},x_{n_k}\big),a\Big) \rightarrow 0,\qquad \dSeq \big(\f(\mfu_{n_k},y_{n_k}),b \big) \rightarrow 0,
	\end{equation}	
	along some subsequence $(n_k)_{k \in \N}$. 	Using an induction argument, we will now 	prove that 	$a = b$ holds. \par 	
	We suppress the dependence on the subsequence, set $m_n^1:= \min(I^n_{1})$ and observe that by \ref{eq.p.fnu.1}
	\begin{equation}\label{eq_ab0}
		\begin{split}
			b_1 &= \lim_{n \rightarrow \infty} f(\mfu_n,y_n)_1 = \lim_{n \rightarrow \infty} \sum_{j \in I^n_{1}} \f(\mfu_n, x_n)_j \\
			&= \lim_{n \rightarrow \infty}\f(\mfu_n, x_n)_{m_n^1} \\
			&\hspace{0.5cm}+ \lim_{n \rightarrow \infty}\ \mathds{1}(|I^n_i| \ge 2) \frac{1}{\f(\mfu_n, x_n)_{m_n^1}}\sum_{j \in I^n_{1}, j >m_n^1} \f(\mfu_n, x_n)_{m_n^1}\f(\mfu_n, x_n)_j\\
			&= \lim_{n \rightarrow \infty} \f(\mfu_n, x_n)_{m_n^1} \le a_1.
		\end{split}
	\end{equation}
	Since $\{I_i^n\}_{i = 1,\ldots,N(y_n)}$ is a partition of $\{1,\ldots,N(x_n)\}$ we find a $\bar i_n \in \{1,\ldots,N(y_n)\}$, such
	that $1 \in I_{\bar i_n}^n$. Again, by \eqref{eq.p.fnu.1}, it follows that 
	\begin{equation}\label{eq_ba}
		\begin{split}
			b_1 &= \lim_{n \rightarrow \infty}\f(\mfu_n, y_n)_1 \ge \lim_{n \rightarrow \infty} \f(\mfu_n, y_n)_{\bar i_n} \\
			&= \lim_{n \rightarrow \infty} \sum_{j \in I^n_{\bar i_n}} \f(\mfu_n, x_n)_j = \lim_{n \rightarrow \infty} \f(\mfu_n, x_n)_1  = a_1
		\end{split}
	\end{equation}	
	and hence $a_1 = b_1$. Let $k \in \N$ and assume that $a_l = b_l$ for all $l = 1,2,\ldots,k$, i.e. 
	\begin{equation}
		\lim_{n \rightarrow \infty} |\f(\mfu_n, y_n)_l - \f(\mfu_n, x_n)_l | = 0. 
	\end{equation} 
	Then, after a suitable reordering, we can assume that $l \in I_l^n$ for all $l = 1,2,\ldots,k$ and all $n$ large enough.
	We can now apply the argument in (\ref{eq_ab0}) and in (\ref{eq_ba}) to get  $b_{k+1} \le a_{k+1}$ and  $b_{k+1} \ge a_{k+1}$. \par  
	Similar arguments show that the other cases hold (see again Lemma \ref{lem_representatives} and its proof).	
\end{proof}

Now, we are ready to prove the main result:

\begin{proof} (Theorem \ref{thm.perfect}) \noindent{\bf (i) - Continuity} \\
We start by proving  $\mfu_n \rightarrow \mfu$ in the Gromov-weak atomic topology implies $\F(\mfu_n) \rightarrow \F(\mfu)$ in the Skorohod topology, i.e. we prove that $\F$ is continuous. \par 
We apply Remark \ref{r.totalMassCont}, Proposition \ref{p.weak.atomic} and Corollary 3.3.2 in \cite{EK86} and get that $F_n \rightarrow F$ in the Skorohod topology, where $F_n(h):=\nu^{2,\mfu_n}([0,h])1(h \ge 0)$ and $F(h):=\nu^{2,\mfu}([0,h])1(h \ge 0)$. 
Take a sequence $h_n > 0$ such that $h_n \rightarrow h > 0$ and assume that 
$F_n(h_n) \rightarrow F(h)$. Let $\bar \e > 0$ and take an $\e > 0$ such that $|F(h+\e) - F(h)| < \bar \e$ and $h+\e$ is a continuity point of $F$.  
It follows that $|F_n(h_n) - F_n(h+\e)| < \bar \e$ for all $n$ large enough and hence, by Lemma \ref{lem_fnu}, for all $\bar \e > 0$ there is an $\e > 0$ such that 
\begin{equation}
\dSeq\Big(\f\big(\mfu_n,h_n\big),\f\big(\mfu_n,h+\e\big)\Big) < \bar \e 
\end{equation}
for all $n$ large enough. 
Since we can choose $\e$ such that $h + \e $ is a continuity point of $F$ and hence of $\f(\cdot,\mfu)$ (see Lemma \ref{l.connect.f.nu}), Lemma \ref{l.conv.contpoint} implies that for all $n$ large enough
\begin{equation}
\dSeq\Big(\f\big(\mfu_n,h_n\big),\f\big(\mfu,h+\e\big)\Big) < \bar \e. 
\end{equation}
Since $\f\big( \mfu,h+\e\big) \rightarrow \f\big(\mfu,h\big)$ for $\e \downarrow 0$, this gives 
\begin{equation}
\dSeq\Big(\f\big(\mfu_n,h_n\big),\f\big(\mfu,h\big)\Big) \rightarrow 0. 
\end{equation}
If $F_n(h_n) \rightarrow F(h-)$ we can use a similar argument to show that 
\begin{equation}
\dSeq\Big(\f\big(\mfu_n,h_n\big),\f\big(\mfu,h-\big)\Big) \rightarrow 0. 
\end{equation}

We can now apply Proposition 3.6.5 in \cite{EK86} (see also Section \ref{sec.Skorohod}) and get $\F(\mfu_n) \rightarrow \F(\mfu)$ in 
the Skorohod topology, when we can show that $0 \le t_n \rightarrow 0$ implies $\dMax(\f(\mfu_n,t_n),\f(\mfu,0)) \rightarrow 0$. \par 
By Lemma \ref{l.Gwa.embedding}, we may assume w.l.o.g. that $\mfu_n = [X,r,\mu_n]$, $\mfu = [X,r,\mu]$ and $\mu_n \rightarrow \mu$ in the 
weak atomic topology and applying Lemma 2.5 in \cite{EKatomic} gives $\dMax(\f(\mfu_n,0), \f(\mfu,0) ) \rightarrow 0$. 
Since $\nu^{2,\mfu}$ is purely atomic (see Remark \ref{rem.nu.atomic}) and again, by Lemma 2.5 in \cite{EKatomic}, it is not 
hard to see that $\nu^{2,\mfu_n}(0,t_n] \rightarrow 0$ (otherwise two atoms would merge in the limit which is prohibited in the weak atomic topology). 
Moreover, by Lemma \ref{lem_restr}, 
\begin{equation}
\dGP(\Cut_{t_n}(\mfu_n),\mfu) \le \dGP(\Cut_{t_n}(\mfu_n),\mfu_n) + \dGP(\mfu_n,\mfu) \rightarrow 0
\end{equation}
and 
\begin{equation}
\begin{split}
\nu^{2,\Cut_{t_n}(\mfu_n)}(\{0\}) &= \nu^{2,\mfu_n}([0,t_n]) \\
&= \nu^{2,\mfu_n}(\{0\})  +\nu^{2,\mfu_n}((0,t_n]) \rightarrow \nu^{2,\mfu}(\{0\}).
\end{split}
\end{equation}
Applying again Lemma \ref{l.Gwa.embedding} and the argument from above gives 
\begin{equation}
 \dMax(\f(\mfu_n,t_n), \f(\mfu,0) ) \rightarrow 0
\end{equation}
and the result follows. \\
\medskip

\noindent {\bf (i) - Perfectness} \\
Let $K \subset \F(\UM)$ be compact, then, because of the continuity of $\F$, $\F^{-1}(K)$ is closed and, as a direct consequence of Proposition \ref{thm_basic_tools}, $\F^{-1}(K)$ is relatively compact in the Gromov-weak 
topology. \par 
If we now take $\mfu_n \in \F^{-1}(K)$ this gives $\mfu_{n_k}\rightarrow \mfu \in \overline{\F^{-1}(K)}$ (where the closure is taken with respect to 
the Gromov-weak topology) along some subsequence (where we suppress this dependence in the following) in the Gromov-weak topology. 
Since $\F(\mfu_n) \in K$, $\{\F(\mfu_n):\ n \in \mathbb N\}$ is relatively compact in the Skorohod topology.
Moreover, since 
\begin{equation}
 \left|\nu^{2,\mfu_n}([0,h]) - \nu^{2,\mfu}([0,h]) \right| \le \dSeq(\f(\mfu_n,h),\f(\mfu,h)), \qquad \text{for all } h > 0,
\end{equation}
we can apply Theorem 3.6.3. in \cite{EK86} together with Proposition \ref{p.weak.atomic} (recall that continuous images of compact sets are compact) 
to get that $\{\nu^{2,\mfu_n}\big|_{[\delta,\infty)}:\ n \in \mathbb N\}$ is relatively compact in the weak atomic topology for all 
continuity points $\delta > 0$ of $\nu^{2,\mfu}([0,\cdot])$. It follows that $\nu^{2,\mfu_n}\big|_{[\delta,\infty)} \rightarrow \nu^{2,\mfu}\big|_{[\delta,\infty)}$
in the weak atomic topology and therefore $\Cut_{\delta}(\mfu_n) \rightarrow \Cut_{\delta}(\mfu)$  in the Gromov-weak atomic topology, 
for all such $\delta > 0$. Note that this also gives 
$\F(\Cut_{\delta}(\mfu_n)) = (\f(\mfu_n,h))_{h \ge \delta} \rightarrow  (\f(\mfu,h))_{h \ge \delta} $ and hence, since 
$\F(\mfu_n)$ is relatively compact, $\F(\mfu_n) \rightarrow \F(\mfu)$ in the Skorohod topology. \par 
It remains to prove that 
\begin{itemize}
\item[(i)] $\lim_{\delta \downarrow 0}\limsup_{n \rightarrow \infty }\nu^{2,\mfu_n}((0,\delta]) = 0$, 
\item[(ii)] $\nu^{2,\mfu_n}(\{0\}) \rightarrow \nu^{2,\mfu}(\{0\})$.
\end{itemize}
Note that the Portmanteau Theorem gives "(ii) $\Rightarrow $ (i)". \par 
In terms of Lemma \ref{l.Gwa.embedding}  we may assume w.l.o.g. that $\mfu_n = [X,r,\mu_n]$, $\mfu = [X,r,\mu]$ with $\mu_n \Rightarrow \mu$ and need to prove $\mu_n \rightarrow \mu$ in the weak atomic topology.  But by Lemma 2.5 in \cite{EKatomic}, this is implied by the convergence $\dMax(\f(\mfu_n,0),\f(\mfu,0)) \rightarrow 0$. \\
\medskip

\noindent {\bf (ii) - Results for the subspace}\\
We start by proving that $\F$ restricted to $\UMI$ is injective. Assume $\F(\mfu) = \F(\mfu')$ for $\mfu = [X,r,\mu],\mfu'=[X',r',\mu'] \in \UMI$ but $\mfu \neq \mfu'$. 
Since $\mfu \in \UMI$ iff $\Cut_h(\mfu) \in \UMI$ and $\mfu = \mfu'$ iff $\Cut_h(\mfu) \in \Cut_h(\mfu') $ for all $h > 0$, we assume w.l.o.g. that $\mfu \in \UMI \cap \UM^a$.
Using the idea in the proof of Theorem \ref{thm_polish} (ii) we may further assume that $\mu$ and $\mu'$ have only finitely many atoms
and we denote the atoms by $(a_i)_{i = 1,\ldots,m}$ and $(a_i')_{i = 1,\ldots,m'}$, $m,m' \in \mathbb N$. \par 
Let $0 =:t_0< t_1 < \ldots < t_n$ be the discontinuity points of $\f(\mfu,\cdot) = \f(\mfu',\cdot)$ and note that these points correspond 
exactly to the points $\{r(x,y):x,y \in\supp(\mu),\ x \neq y\}$. Moreover, since  $\f(\mfu,0) = \f(\mfu',0)$ we get (up to reordering) 
$(a_i)_{i = 1,\ldots,m} =(a_i')_{i = 1,\ldots,m'}$.  When we now apply Lemma \ref{l.Ielements}
we find for all $t_i$ pairwise disjoint sets $(A_{i}^k)_{k = 1,\ldots,K}$, where $K$ is the number of the non zero elements of $\f(\mfu,t_i)$ such that 
\begin{equation}
 \f(\mfu,t_i)_k = \sum_{j \in A_{i}^k} a_j,\qquad k = 1,\ldots,K
\end{equation}
and the choice of $(A_{i}^k)_{k = 1,\ldots,K}$ is unique. This uniqueness allows us to define the ultra-metric $\tilde r$ on $\{1,\ldots,m\}$ by
\begin{equation}
  \tilde r(x,y) = t_i \quad :\Longleftrightarrow \quad 
  \min\left\{j \in\{1,\ldots,n\}, \exists k \text{ s.t. }x,y \in A^k_j\right\} = i
\end{equation}
and we get
\begin{equation}
 \left[\{1,\ldots,m\},\tilde r,\ \sum_{i = 1}^m a_i \delta_i \right] = \mfu
\end{equation}
(compare also Remark \ref{rem.reconstruction}). Since all the quantities that where necessary for the above construction are completely determined by the values of $\f(\mfu,\cdot)$, we have $\mfu = \mfu'$ and get the 
injectivity of $\F\big|_\UMI$. \\

Since $\F\big|_\UMI$ is continuous, it remains to prove that $\F(\UMI)\ni f_n \rightarrow f \in \F(\UMI)$ implies 
$\mfu_n := \F^{-1}(f_n) \rightarrow \F^{-1}(f) =: \mfu$. \par 
Since  $\{f_n:\ n \in \mathbb N\}$ is relatively compact, we can apply part (i) of this theorem and 
get $\{\mfu_{n}: n \in \mathbb N\}$ is relatively compact and hence satisfies $\mfu_{n_k} \rightarrow \tilde \mfu \in \UM$ along some subsequence. It follows that 
\begin{equation}
 \F(\mfu_{n_k}) = f_{n_k} \rightarrow f = \F(\tilde \mfu) \in\F(\UMI).
\end{equation}
Since  $\mfu \in \UMI$ if and only if $\F(\mfu) \in \F(\UMI)$ (this is exactly the construction from above), we get  $\tilde \mfu = \F^{-1}(f)$.
\end{proof} 

\begin{remark}
 As we have seen in the proof of part (ii), the following holds:  $\mfu \in \UMI$ if and only if $\F(\mfu) \in \F(\UMI)$.
\end{remark}

\begin{proof}(Proof of  Theorem \ref{thm.subspace.homeomorphism})
We use the ideas of Theorem \ref{thm.perfect} (ii), take $\mfu=[X,r,\mu] \in \UMB \cap \UMI$ and assume $\mu$ has only finitely many atoms $(a_i)_{i = 1,\ldots,m}$. 
Denote the discontinuity points of $\nu^{2,\mfu}[0,\cdot]$ by $t_0:=0<t_1<\ldots<t_n$ and observe that, by Lemma \ref{l.connect.f.nu}, 
these points correspond to $\{r(x,y):x,y\in \supp(\mu),\ x\neq y\}$. In particular $m = n+1$. \par 

Take $\mfu'=[X',r',\mu'] \in \UMB \cap \UMI$ and assume $\nu^{2,\mfu} = \nu^{2,\mfu'}$, then, as in the proof of Theorem \ref{thm.perfect}, the set of discontinuity points of
$\nu^{2,\mfu'}[0,\cdot]$ coincides with $t_0:=0<t_1<\ldots<t_n$. Moreover, since $\nu^{2,\mfu}([0,t_n]) = \nu^{2,\mfu'}([0,t_n]) $ we get $\f(\mfu,t_n) = \f(\mfu',t_n)$. \par 
Assume now, that $\f(\mfu,t_k) = \f(\mfu',t_k)$ for all $k = n,n-1,\ldots,K+1$ and some $K \le  n-1$ and define $a^K_i:=\f(\mfu,t_{K})_i$, $\hat a^K_i:=\f(\mfu',t_{K})_i$.
If $\f(\mfu,t_{K}) \neq \f(\mfu',t_{K})$, then, by definition of $\UMB \cap \UMI$, there is a $i^\ast, i_1,i_2$ and $j^\ast, j_1,j_2$ such that 
\begin{equation}
 a^{K+1}_l = \left\{ \begin{array}{ll}
                      a^K_l,&\quad l < i^\ast,\\
                      a^K_{i_1} + a^k_{i_2},&\quad l = i^\ast,\\
                      a^K_{l-1},&\quad l \in \{i^\ast+1,\ldots,i_1\},\\
                      a^K_l,&\quad l \in\{i_1+1,\ldots,i_2-1\},\\
                      a^K_{l+1},&\quad l > i_2.
                     \end{array}\right.
\end{equation}
and 
\begin{equation}
\hat a^{K+1}_l = \left\{ \begin{array}{ll}
                      \hat a^K_l,&\quad l < j^\ast,\\
                      \hat a^K_{j_1} + a^k_{j_2},&\quad l = j^\ast,\\
                      \hat a^K_{l-1},&\quad l \in \{j^\ast+1,\ldots,j_1\},\\
                      \hat a^K_l,&\quad l \in\{j_1+1,\ldots,j_2-1\},\\
                      \hat a^K_{l+1},&\quad l > j_2.
                     \end{array}\right.
\end{equation}
and for a given $\mfu \in \UMB \cap \UMI$, $i^\ast,i_1,i_2$ are unique. By assumption, we therefore need to prove that $\{a_{i_1}^K,a_{i_2}^K\} = \{\hat a_{j_1}^K,\hat a_{j_2}^K\} $. \par 
If $i^\ast = j^\ast $, then, since $\nu^{2,\mfu} = \nu^{2,\mfu'}$,
\begin{equation}
 \begin{split}
  (I) & a^K_{i_1}+a^K_{i_2} = \hat a^K_{j_1}+\hat a^K_{j_2} \quad \Leftrightarrow \quad a^{K+1}_{i^\ast} = a^{K+1}_{j^\ast}, \\
  (II) & \left(a^K_{i_1}\right)^2+ \left(a^K_{i_2}\right)^2 + \left(a^{K+1}_{j^\ast}\right)^2 = \left(\hat a^K_{j_1}\right)^2+\left(\hat a^K_{j_2}\right)^2 + \left(a^{K+1}_{i^\ast}\right)^2 \\
 \end{split}
\end{equation}
and the result follows. Hence, we assume $i^\ast  \neq j^\ast $, (which implies $a^{K+1}_{i^\ast} \neq a^{K+1}_{j^\ast}$). 
Note that by Theorem \ref{thm.perfect} (ii), we have 
\begin{equation}
 \Cut_{t_{K+1}}(\mfu) = \Cut_{t_{K+1}}(\mfu') =:\left[\{1,\ldots,n-K\},r,\sum_{l = 1}^{n-K} a_l^{K+1}\delta_l\right].
\end{equation}
We assume w.l.o.g. that $i_1 = 1,i_2 = 2,j^\ast = 3$ and define the ultra-metric $\tilde r^1$ on $\{1,2,3\}$ by 
\begin{align}
 \tilde r^1(1,2) = t_{K+1}-t_k =:s^1,\ \quad \tilde r^1(1,3) = r(i^\ast,j^\ast) + t_{K+1}-t_k=:s^2.
\end{align}
On the other hand we assume, w.l.o.g. $j_1 = 4,j_2 = 5,i^\ast = 6$ and define the ultra-metric $\tilde r^2$ on $\{4,5,6\}$ by 
\begin{equation}
 \tilde r^2(4,5) = t_{K+1}-t_k =s^1,\ \quad \tilde r^2(4,6) = r(i^\ast,j^\ast) + t_{K+1}-t_k=s^2.
\end{equation}
By assumption $\nu^{2,\mfu^1} = \nu^{2,\mfu^2}$, where 
\begin{equation}
 \mfu^1 = \left[\{1,2,3\},\tilde r^1,a^{K}_{i_1} \delta_1 +a^{K}_{i_2} \delta_2 +a^{K+1}_{j^\ast} \delta_3  \right]
\end{equation}
and 
\begin{equation}
 \mfu^2 = \left[\{4,5,6\},\tilde r^2,a^{K}_{j_1} \delta_1 +a^{K}_{j_2} \delta_2 +a^{K+1}_{i^\ast} \delta_3  \right].
\end{equation}
For simplicity, set $x:=a^{K}_{i_1} $, $y:=a^{K}_{i_2} $, $z:=a^{K+1}_{j^\ast} $, then the following holds: 
\begin{equation}
 \begin{split}
  b_0:=\nu^{2,\mfu}(\{0\}) &= x^2+y^2+z^2,\\
  b_1:=\nu^{2,\mfu}([0,t_1]) &= (x+y)^2+z^2 \\
  b_2:=\nu^{2,\mfu}([0,t_2]) &= (x+y+z)^2.
 \end{split}
\end{equation}
 Now, up to reordering and with respect to the constraints on the $a_i$, there is only one possible solution to the above system of equations, given by
\begin{equation}
 \begin{split}
 x&:= \frac{\sqrt{2\sqrt{b_2}\cdot \sqrt{2b_1-b_2}+2b_1} + \sqrt{2\sqrt{b_2} \cdot \sqrt{2b_1-b_2}+8 b_0-6b_1}}{4},\\
 y&:= \frac{\sqrt{2\sqrt{b_2}\cdot \sqrt{2b_1-b_2}+2b_1} - \sqrt{2\sqrt{b_2} \cdot \sqrt{2b_1-b_2}+8 b_0-6b_1}}{4},\\
 z&:=\frac{\sqrt{b_2}-\sqrt{2b_1-b_2}}{2}
 \end{split}
\end{equation}
and the result follows. 
\end{proof}

\subsection{Proof of Proposition \ref{p.UMI.as}}

Note that by definition of $\UMI$, $\mfu \in \UMI$ if and only if $\Cut_h(\mfu) \in \UMI$ for all $h > 0$, $\Cut_h(\mfu) \in \UMI$ implies $\Cut_{h'}(\mfu)\in \UMI $
whenever $h' > h$. \par 
Fix an $h > 0$ and let $(a_i^h)_{i = 1,\ldots,N_h} = \f(\mathcal U_t,h)$. Since $(a_i^h)_{i = 1,\ldots,N_h}$ are, conditioned on $N_h$, independent 
and $\mathcal L(a_i^h)$ are absolutely continuous with respect to the Lebesgue measure for all $i$, we get 
\begin{equation}
\begin{split}
 P&\left(\exists I \subset\{1,\ldots,N_h\}:\ \sum_{i \in I} a_i = x\Big|N_h\right) \\
&=  P\left(\exists I \subset\{1,\ldots,N_h\}:\ a_{\min(I)} = x- S_I\Big|N_h\right) = 0,
\end{split}
\end{equation}
for all $x \in \mathbb R$, where $S_I = \sum_{i \in I\setminus\{\min(I)\}} a_i$. We can now apply Lemma \ref{l.Ielements} and get

\begin{equation}
 \begin{split}
  P\left(\mathcal U_t \in \UMI\right) &= P\left(\Cut_h(\mathcal U_t) \in \UMI,\ \forall h > 0\right) \\
	&= \lim_{h \downarrow 0} P\left(\Cut_h(\mathcal U_t) \in \UMI\right) \\
	&= 1-\lim_{h \downarrow 0}P\left(\Cut_h(\mathcal U_t) \not\in \UMI\right) \\
	&\ge 1- \lim_{h \downarrow 0} E\left[ E\left[ 1\left (\mathcal L\left((a_i^h)_{i = 1,\ldots,N_h}\right) \not\ll \lambda^{\otimes N_h}\right) \Big|N_h\right] \right]\\
  &= 1.
 \end{split}
\end{equation}

\section{Proof of Theorem \ref{thm_polish} (c)}\label{sec.p.polish.c}
\def\NUM{\mathcal N}
\def\Num{\mathfrak{n}}
As we have seen in Remark \ref{r.num.lower}, the number of balls map has some nice properties.
In fact these properties allow us to proof our result and we start by formalizing this remark.  

\begin{definition}
Define the number of balls map $\NUM:\UMc \to D((0,\infty),\mathbb N)$, $\mfu \mapsto (h \mapsto \min\{i \in \mathbb N: \f(\mfu,h)_{i +1} = 0\})$ and $\Num(\f(\mfu,h)) := \NUM(\mfu)_h$. 
\end{definition}

\begin{remark}
(1) In terms of Lemma \ref{lem_representatives}, $\NUM(\mfu)_h = \Num(\f(\mfu,h))  = \num(h)$. \par
(2) Using Lemma \ref{l.F.cadlag}, it is not hard to see that the range of the map is a subset of the Skorohod space. \par 
(3) Since $\mfu \in \UMc$, $\NUM(\mfu)_h < \infty$ for all $h > 0$.  
\end{remark}

\begin{lemma}
$\mfu \in \UMB\cap \UMc$ if and only if $\mfu \in \UMc$ and $\Num(\f(\mfu,h-))-\Num(\f(\mfu,h)) \le 1$ for all $h > 0$. 
\end{lemma}

\begin{proof}
Assume that $\mfu = [X,r,\mu]$ and that there is an $h > 0$ such that $\Num(\f(\mfu,h-))-\Num(\f(\mfu,h))  \ge 2$. Then, by the argument in the second part of Lemma \ref{lem_representatives}, there  are at least three points $x,y,z \in X$ such that 
\begin{equation}
\min(r(x,y),r(x,z),r(y,z)) \ge h
\end{equation}
and a $\tau \in X$ such that 
\begin{align}
\mu(\bar B(\tau,h) \cap B(x,h)) &= \mu(B(x,h)), \\
\mu(\bar B(\tau,h) \cap B(y,h)) &= \mu(B(y,h)), \\
\mu(\bar B(\tau,h) \cap B(z,h)) &= \mu(B(z,h)), 
\end{align}
which contradicts the definition of $\UMB$. 
\end{proof}

\begin{proof} (Theorem \ref{thm_polish} (c))
Let $\mfu_n \in \UMB\cap \UMc$ with $\mfu_n \rightarrow \mfu \in \UMc$ in the Gromov-weak atomic topology. Then $\F(\mfu_n)\rightarrow \F(\mfu)$, by Theorem \ref{thm.perfect},  and hence $\f(\mfu_n,\delta) \rightarrow \f(\mfu,\delta)$ with respect to $\dSeq$ for all continuity points $\delta > 0$ of $\f(\mfu,\cdot)$. Therefore, we can find for all $\eps > 0$ a $K \in \mathbb N$ such that 
\begin{equation}
\sup_{n \in \mathbb N} \sum_{i \ge K}\f(\mfu_n,\delta)_i \le \eps.  
\end{equation}
Now, as we have seen in the proof of Lemma \ref{l.conv.contpoint}, for all $\eps > 0$ small enough and all continuity points $h > 0$ of $\f(\mfu,\cdot)$, there is a $N \in \mathbb N$ such that $\Num(\f(\mfu_n,h)^\eps) \equiv \Num(\f(\mfu,h))$ for all $n \ge N$, where 
\begin{equation}
\f(\mfu_n,\delta)^\eps_i :=  \f(\mfu_n,\delta)_i \mathds{1}(\f(\mfu_n,\delta)_i > \eps),\qquad i \in \mathbb N.
\end{equation}
Finally observe that by the convergence in the Skorohod topology and since $(\f(\mfu,h))_{h \ge \delta}$ has only finitely many jumps for all $\mfu \in \UMc$ and $\delta > 0$, two jumps of size larger than $\eps$ (with respect to $\dSeq$) are uniformly separated, say by $\eta$ (see Theorem 3.6.3 in \cite{EK86}), and therefore, 
\begin{equation}
\begin{split}
\Num(\f(\mfu,h-)) - \Num(\f(\mfu,h)) &= \Num(\f(\mfu,h-\eta/3))- \Num(\f(\mfu,h+\eta/3)) \\
&= \Num(\f(\mfu_n,h-\eta/3)^\eps) - \Num(\f(\mfu_n,h+\eta/3)^\eps) \\
&\le 1.
\end{split}
\end{equation}
for all $n$ large enough and all $h > 0$, where we assumed that $\eps$ is small enough and that $h \pm \eta/3$ is a continuity point of $\f(\mfu,\cdot)$. 
\end{proof}

\section{Proofs for Section \ref{sec.FV.result}}

We start by proving that the convergence of the marginals of the tree-valued Moran model to the tree-valued Fleming-Viot process 
also holds, when the space $\UM$ is equipped with the Gromov-weak atomic topology. Then we apply this result to prove our main result. 

\subsection{Proof of Proposition \ref{p.convMM}}

In order to see that $\mathcal U^N_t \Rightarrow \U_t$, when $\UM$ is equipped with the Gromov-weak atomic topology, we will prove:

\begin{lemma}\label{lem.convMM}
 Under the assumptions of Proposition \ref{p.convMM} one has weak convergence of  $\F(\mathcal U^N_t)$ for all $t \ge 0$. 
\end{lemma}

Once we have shown this Lemma, we can apply Theorem \ref{thm.perfect} to get relative compactness and hence convergence of $\U_t^N$ to $\U_t$, that is
Proposition \ref{p.convMM}.

\begin{remark}
 Note that the above implies $\F(\mathcal U^N_t) \Rightarrow \F(\mathcal U_t)$.
\end{remark}

In order to prove Lemma \ref{lem.convMM} we need the following.

\subsubsection{Connection to the Kingman-coalescent}\label{sec.ConKing}

Here we give the connection of the tree-valued Moran model and the Kingman-coalescent (see the proof of Proposition 6.15 in \cite{grieshammer2016}). 
For the definition and properties of the Kingman-coalescent we refer to \cite{B} and \cite{Ber}. \par  

For fixed $t \ge 0$, we set $A_h(i):=A_{t-h}(i,t)$ (see \eqref{eq.ancestors}), $0 \le h \le t$ and $[N]:= \{1,\ldots,N\}$. 
Then $\{A_h(i):\ i \in [N]\}$ can be described as a family of processes in $[N]^N$ that starts in $A_0(i) = i$ and has the following dynamic: 
Whenever $\eta^{i,j}(\{t-h\}) = 1$ for some $i,j \in [N]= I_N$ we have the following transition: 
\begin{equation}
A_{h-}(k) \rightarrow  A_{h}(k) = i, \qquad \forall k \in \{l \in [n]: A_{h-}(l) = j\}.
\end{equation}
It is now straightforward to see that the time it takes to decrease the number of different labels, $|\{A_h(i):\ i \in [N]\}|$, by $1$, given there are $k$ different labels, 
is exponential distributed with parameter $\binom{k}{2}$ and that the two labels (the one that replaces and the one that is replaced) are sampled uniformly 
without replacement under all existing labels. If we define  
\begin{equation}
\kappa_i(h) = \{j\in [N]:\ A_{h}(j) = A_{h}(i)\},
\end{equation}
this implies $\kappa^N = (\{\kappa_1(h),\ldots,\kappa_N(h)\})_{0\le h \le t}$ is a Kingman $N$-coalescent (up to time $t$). 
If we now set 
\begin{equation}
\mathcal V_t^N = \left[\{1,\ldots,N\},r^\kappa_t,\frac{1}{N}\sum_{k = 1}^N \delta_k\right],
\end{equation}
where 
\begin{equation}
r^\kappa_t(i,j) = \left\{\begin{array}{ll}
\inf\{h \ge 0|\ \exists k: i,j \in \kappa_k^N(h)\},& \text{if } \exists k: \ i,j \in \kappa_k^N(t-),\\
r_0(i,j),&\text{otherwise},
\end{array}\right.
\end{equation}
and $r_0$ is the ultra-metric given in the definition of $\mathcal U_t^N$ (see \eqref{eq.ultrametric}), then the above implies 
\begin{lemma}
 $\mathcal L(\mathcal V_t^N) = \mathcal L(\mathcal U_t^{N})$.
\end{lemma}

Now, note that the Kingman-coalescent $\kappa = (\kappa_t)_{t \ge 0}$ satisfies the consistency relation: 
\begin{itemize}
 \item[(C)] $\kappa\big|_N := (\kappa(t)\big|_N)_{t \ge 0}$ is a Kingman $N$-coalescent started from $\kappa(0)\big|_N$,  where for $\pi = \{\pi_1,\pi_2,\ldots\} \in \mathcal P(\mathbb N)$ (set of partitions of $\mathbb N$)
 we define $\pi\big|_N \in \mathcal P([N])$ (set of partitions of $[N]$) as the element induced by $\pi_1\cap[N],\pi_2\cap[N],\ldots$. 
\end{itemize}
Note further, that the law of the Kingman-coalescent is determined by this property and the initial configuration. 
As a consequence, we may assume: 
\begin{assumption}
Our processes are defined on a probability space, where the Kingman $N$-coalescents are coupled (for different $N$) 
in such a way that they are restrictions of an underlying Kingman coalescent 
$\kappa = (\kappa_t)_{t \ge 0}$, i.e. we assume that $(\kappa^N)_{N \in \mathbb N} :=(\kappa\big|_N )_{N \in \mathbb N}$.
\end{assumption}

\subsubsection{Proof of Lemma \ref{lem.convMM}}\label{sec.conn.King}

First note that by the construction in \ref{sec.ConKing} we have that $\f(\U^N_t,h)_{0 \le h \le t}$ is given by the decreasing reordering 
of 
\begin{equation}
\left(\frac{1}{N} |\kappa^N_i(h)|\right)_{i \in [N]} = \left(\frac{1}{N} |\kappa_i(h)\cap [N]|\right)_{i \in \mathbb N},
\end{equation}
where $\kappa^N(0) = \{\{1\},\ldots,\{N\}\}$ and $|\cdot |$ denotes the number of elements in $\cdot$. 
Next, we define for $A \subset \mathbb N$
\begin{equation}
 |A|_f := \lim_{N \rightarrow \infty}\frac{|A \cap [N] |}{N},
\end{equation}
if it exists, and call $|A|_f$ in this case the {\it asymptotic frequency} of $A$. 

\begin{proof} (Lemma \ref{lem.convMM}) We first note that the Kingman-coalescent at a time $t> 0$ forms an exchangeable random partition and therefore possesses 
assymptotic frequency almost surely (see \cite{B}). That is, when we denote by $(\tau_k)_{k \in \mathbb N}$ the first times of the Kingman-coalescent to  have 
$k$ blocks, i.e. $\tau_k := \inf\{t > 0: |\kappa(t)| = k\}$, then $|K_i(k)|_f:= |\kappa_i(\tau_k)|_f$ exists for all $i \in \mathbb N$ almost surely. 
We denote by $K^\downarrow (k)$ the decreasing reordering of the block frequencies $(|K_i(k)|_f)_{i = 1,\ldots,k}$ and  
define $f:(0,\infty) \to \mathcal S^\downarrow$ by 
\begin{equation}
 f(t):= \left\{\begin{array}{ll}
                K^\downarrow (1),&\ \text{when } t \ge \tau_1,\\
                K^\downarrow (k),&\ \text{when } t \in [\tau_{k},\tau_{k-1}),\ k \ge 2.
               \end{array}\right.
\end{equation}
Note that $\f(\U^N_t,t) = \f(\U^N_t,t-)$ almost surely, hence, by construction, $f \in D((0,\infty), \mathcal S^\downarrow)$ and 
\begin{equation}
 (\f(\U^N_t,h))_{\delta \le h \le t} \rightarrow (f(h))_{\delta \le h \le t} 
\end{equation}
in the Skorohod topology almost surely for all $\delta > 0$, where $\mathcal S^\downarrow $ is equipped with $\dSeq$ (recall that the Kingman-coalescent comes down from infinity, i.e. 
there are only finitely many elements of $f(h)$ that are non-zero). 

Moreover, by definition, $\Cut_t(\U_t^N) = [[0,1],r_0,\mu_t^N]$ and $\Cut_t(\U_t) = [[0,1],r_0,\mu_t]$, where 
 \begin{equation}
  \mu_t^N \stackrel{d}{=} \sum_{i \in \mathbb N} \f(\U_t^N,t)_i \delta_{V_i}
 \end{equation}
 and 
  \begin{equation}
  \mu_t \stackrel{d}{=} \sum_{i \in \mathbb N} f(t)_i \delta_{V_i}
 \end{equation}

and  $\mu_t^N \Rightarrow \mu_t$. By the coupling of the Kingman-coalescent and the Kingman-$N$-coalescent, we know further, 
that for $N$ large enough, the number of non-zero entries of $\f(\U_t^N,t)$ (which corresponds to the number of blocks in a 
Kingman-$N$-coalescent) equals the number of non-zero entries of $f(t)$. Combining this observation with the fact that 
$\sum_{i \in I} \f(\U^N_t,t)_i \rightarrow \sum_{i \in I} f(t)_i$, for all $I \subset \mathbb N$, $|I| < \infty$, and 
Lemma \ref{lem_representatives}, we get 
\begin{equation}
 \f(\U^N_t,h)_{h \ge t} \Rightarrow f(h)_{h \ge t}.
\end{equation}

Finally observe that for $\eps > 0$ and $\delta > 0$,
\begin{equation}
\begin{split}
 \limsup_{N \rightarrow \infty}  P&\left(\sup_{h \in [0,\delta] } \dMax(\f(\U^N_t,h), \f(\U^N_t,0)) \ge \eps \right) \\
 &=  \limsup_{N \rightarrow \infty}  P\left(\left| \f(\U^N_t,\delta)_1 - \frac{1}{N}\right|  \ge \eps \right) \\
 &\le   \limsup_{N \rightarrow \infty}  P\left(\f(\U^N_t,\delta)_1  \ge \eps \right) \\
 &\le  P\left(f(\delta)_1 \ge \eps \right). 
\end{split}
\end{equation}
Note that $K^\downarrow (n)$ is the decreasing rearrangement of a random variable that is uniform distributed on the simplex (see again \cite{B}), i.e. 
 $K^\downarrow (n) = \max(X_1,\ldots,X_n)/\sum_{i = 1}^n X_i$, where $X_1,\ldots,X_n$ are independent exponential-$1$-distributed. We can now apply
the strong law of large numbers together with a simple calculation to show that for $\delta > 0$ small enough
\begin{equation}
\begin{split}
 \limsup_{N \rightarrow \infty}  P&\left(\sup_{h \in [0,\delta] } \dMax(\f(\U^N_t,h), \f(\U^N_t,0)) \ge \eps \right) \le \eps
\end{split}
\end{equation}
(in fact one could prove that $E\left[f(\delta)_1\right]= \frac{1}{k}\sum_{l = 1}^k \frac{1}{l}$ - see for example \cite{simplex}).
\end{proof}

\subsection{Proof of Theorem \ref{thm.FV.neutral}}

Note that the tree-valued Fleming-Viot process
$(\mathcal U_t)_{t \ge 0}$ takes values in the space of compact ultra-metric measure spaces $\UMc$ for all $t > 0$ (see Proposition 2.11 in \cite{GPW13}).
Fix a $t > 0$. It follows that for all $h > 0$ there is an almost surely finite random variable $N_h \in \mathbb N$ such that $\Cut_h(\mathcal U_t) = [X,r,\mu_h]$ satisfies $|\{x\in X| \mu_h(\{x\}) > 0\}| = N_h$ 
almost surely.  Since the reordering of atoms of $\mu_h$ is distributed as the 
decreasing rearrangement of a random variable that is uniformly distributed on the simplex (see \cite{B}; compare also the proof of Lemma \ref{lem.convMM}), 
the result follows by Proposition \ref{p.UMI.as}, once we have shown that $\U_t \in \UMB$ almost surely.  But, since $\U^N_t \in \UMB$ almost surely (recall Section \ref{sec.conn.King}), 
$\UMB$ is closed in the Gromov-weak atomic topology (see Theorem \ref{thm_polish}), and $\U^N_t \Rightarrow \U_t$, the result follows by the Portmanteau theorem.

\begin{appendix}
\section{Skorohod topology}\label{sec.Skorohod}

The space $D((0,\infty),\mathcal S^\downarrow)$ is equipped with the Skorohod topology, which is induced by the following metric 
\begin{equation}
\begin{split}
\dSK(f,g) := \dSKSeq(f,g) + \dSKMax(f,g)
\end{split}
\end{equation}
where 
\begin{align}
\dSKSeq(f,g)&:=\inf_{\lambda \in \Lambda} \left( \gamma(\lambda) \vee \int_0^\infty e^{-u-\frac{1}{u}} \rho_1(f,g,\lambda,u) du \right)\\
\dSKMax(f,g)&:=\inf_{\lambda \in \Lambda} \left( \gamma(\lambda) \vee \int_0^\infty e^{-u} \rho_2(f,g,\lambda,u) du \right),
\end{align}
with
\begin{align}
\rho_1(f,g,\lambda,u) &:= \sup_{t \ge 0}\dSeq(f(t \wedge u),g(\lambda(t) \wedge u))\wedge 1, \\
\rho_2(f,g,\lambda,u) &:=\sup_{t \ge 0}\dMax(f(t \wedge u),g(\lambda(t) \wedge u))\wedge 1
\end{align}
and $\Lambda$ is the set of strictly increasing surjective Lipschitz continuous functions $\lambda:[0,\infty) \to [0,\infty)$ such that 
\begin{equation}
\begin{split}
\gamma(\lambda) := \text{ess~sup}_{s> t \ge 0}\left|\log\left(\frac{\lambda(s)-\lambda(t)}{s-t} \right)\right| < \infty.
\end{split} 
\end{equation}
Here 
\begin{align}
\dSeq(x,y) &:= \sum_{i = 1}^\infty |x_i-y_i|,\\
\dMax(x,y) &:= \max_{i \in \mathbb N} |x_i-y_i|.
\end{align}

We note that the idea for the definition of the first term of $\dSK$ follows the same idea as in the case where one wants to 
include cadlag functions defined on $\mathbb R_+$ (and not only on compact intervals) and one can use the same techniques as 
for example in Section 3 of \cite{EK86} to prove that 

\begin{proposition}
\item[(i)] The space $D((0,\infty),\mathcal S^{\downarrow})$ equipped with $\dSK$ is a complete separable metric space. 
\item[(ii)] The characterization of convergence in this topology is analogue to the characterization given in Proposition  3.6.5 in \cite{EK86}: \par 
 Let $f,f_1,f_2,\ldots \in D((0,\infty),\mathcal S^{\downarrow})$. Then $\lim_{n \rightarrow \infty}\dSK(f_n,f) = 0$
 if and only if whenever $t,t_1,t_2,\ldots \in [0,\infty)$, and $\lim_{n \rightarrow \infty} t_n = t$, the following conditions holds.
 \begin{itemize}
  \item[(a)] If $t > 0$ then 
  $\lim_{n \rightarrow \infty} \dSeq(f_n(t_n),f(t)) \wedge \dSeq(f_n(t_n),f(t-)) = 0$ and if 
  $t = 0$, then $\lim_{n \rightarrow \infty} \dMax(f_n(t_n),f(t))  = 0$.
  \item[(b)] If $t > 0$, $\lim_{n \rightarrow \infty} \dSeq(f_n(t_n),f(t))= 0$ and $s_n \ge t_n$ for each $n$, and 
  $\lim_{n \rightarrow \infty}s_n = t$, then $\lim_{n \rightarrow \infty} \dSeq(f_n(s_n),f(t)) = 0$.
  \item[(c)] If $t > 0$, $\lim_{n \rightarrow \infty} \dSeq(f_n(t_n),f(t-))= 0$ and $0 \le s_n \le t_n$ for each $n$, and 
  $\lim_{n \rightarrow \infty}s_n = t$, then $\lim_{n \rightarrow \infty} \dSeq(f_n(s_n),f(t-)) = 0$.
 \end{itemize}
\end{proposition}

\end{appendix}

\newpage


\end{document}